\documentclass[smallextended]{svjour3}

\usepackage{a4wide, amstext, amsmath, amssymb, graphicx, array, epsfig, psfrag}
\usepackage{stmaryrd}
 \usepackage{color}

\newcommand{\mcK}{\mathcal{K}}
\newcommand{\mcE}{\mathcal{E}}
\newcommand{\mcF}{\mathcal{F}}
\newcommand{\mcN}{\mathcal{N}}
\newcommand{\mcA}{\mathcal{A}}
\newcommand{\mcV}{\mathcal{V}}
\newcommand{\mcW}{\mathcal{W}}

\newcommand{\tn}{|\mspace{-1mu}|\mspace{-1mu}|}

\newcommand{\bfn}{\boldsymbol{n}}
\newcommand{\bfu}{\boldsymbol{u}}
\newcommand{\bfg}{{\boldsymbol{g}}}
\newcommand{\bff}{\boldsymbol{f}}
\newcommand{\bfv}{\boldsymbol{v}}
\newcommand{\bfw}{\boldsymbol{w}}
\newcommand{\bft}{\boldsymbol{t}}
\newcommand{\bfI}{\boldsymbol{I}}
\newcommand{\bfa}{\boldsymbol{a}}
\newcommand{\bfx}{\boldsymbol{x}}

\newcommand{\ldb}{\left\llbracket}
\newcommand{\rdb}{\right\rrbracket}
\newcommand{\bfV}{\boldsymbol{V}}

\newcommand{\bfeps}{\boldsymbol{\epsilon}}
\newcommand{\bfphi}{\boldsymbol{\phi}}

\newcommand{\bfvarphi}{\boldsymbol{\varphi}}
\newcommand{\lp}{\left(}
\newcommand{\rp}{\right)}
\newcommand{\bfzero}{\boldsymbol{0}}
\newcommand{\supp}{\operatorname{supp}}

\newcommand{\omegain}{\omega_{h,i}}
\newcommand{\omegahi}{\Omega_{h,i}}
\newcommand{\omegahO}{\Omega_{h,1}}
\newcommand{\omegahT}{\Omega_{h,2}}

\newtheorem{rem}{Remark}

\numberwithin{lemma}{section}
\numberwithin{theorem}{section}
\numberwithin{equation}{section}
\numberwithin{figure}{section}
\numberwithin{table}{section}

\title{A cut finite element method for a Stokes interface problem} 
\author{Peter Hansbo \and Mats~G.~Larson \and Sara Zahedi}
 \institute{P. Hansbo \at Department of Mechanical Engineering, J\"onk\"oping University, SE-551~11 J\"onk\"oping, Sweden \\
  \email{peter.hansbo@jth.hj.se}
\and
M.~G.~Larson \at Department of Mathematics, Ume{\aa} University, SE--901~87~~Ume{\aa}, Sweden.\\ \email{mats.larson@math.umu.se}
\and
S. Zahedi \at Department of Information Technology, Uppsala University, Box~337, SE--751~05~Uppsala, Sweden. \\ 
 \email{sara.zahedi@it.uu.se}
}
\authorrunning{P.~Hansbo, M.~G.~Larson, S.~Zahedi}

\begin{document}
\maketitle
\begin{abstract}
We present a finite element method for the Stokes equations involving two immiscible incompressible fluids with different viscosities and with surface tension. The interface separating the two fluids does not need to align with the mesh. We propose a Nitsche formulation which allows for discontinuities along the interface with optimal a priori error estimates. A stabilization procedure is included which ensures that the method produces a well conditioned stiffness matrix independent of the location of the interface.
\end{abstract}

\section{Introduction}
A large number of real world phenomena exhibit strong or weak discontinuities. The application we have in mind is multiphase flow with kinks in the velocity field and jumps in the pressure field as  well as in physical parameters, such as viscosity, across interfaces that evolve with time and may undergo topological changes.  When simulating such phenomena,  discontinuities can occur anywhere relative to a fixed background mesh. Unfortunately, standard finite element methods, as well as finite difference schemes, do not accurately model discontinuities that are not \emph{a priori}\/ fitted to the mesh. However,  letting the mesh conform to the interfaces requires remeshing as these evolve with time, and leads to significant complications when topological changes such as drop-breakup or coalescence occur.  Numerous strategies have been proposed to handle these difficulties.

A common strategy has been to regularize the discontinuities~\cite{BKZ92}. However, this strategy has the drawback that it gives reduced accuracy near the interfaces and consequently requires a very fine mesh in these regions. Methods that allow for discontinuities along interfaces that do not align with the mesh, and hence avoid both regularization and remeshing processes, have become highly attractive and significant efforts have been directed to their development, see e.g.~\cite{ASB10,BBH09,CB03,GrRe07,LI06}. In the finite element framework the extended finite element method (XFEM), where the finite element space is enriched so that discontinuities can be captured~\cite{FB10}, has become a popular alternative. However, in XFEM the conditioning of the problem is sensitive to the position of the interface. 
Whenever the interface cuts an element in such a way that the ratio between the areas/volumes on one and the other side of the interface becomes very large, the system may become ill-conditioned. In such cases, iterative linear solvers may breakdown. {For unsteady problems when the interface moves across a fixed background mesh such situations occur when the interface moves into new elements}. In~\cite{R08}, this problem is addressed by neglecting basis functions in the XFEM space that have very small support and may cause ill-conditioning. The criterion for the selection of  basis functions to neglect  then has to be chosen carefully so that  accuracy is not lost.

An alternative to the XFEM approach is based on an extension of Nitsche's method~\cite{Nit} for the weak enforcement of essential boundary conditions. This approach was first proposed for an elliptic interface problem in~\cite{HaHa02} and later for a Stokes interface problem in~\cite{BBH09}. The idea is to construct the discrete solution from separate solutions defined on each subdomain separated by the interface and at the interface enforce the jump conditions weakly using a variant of Nitsche's method. By choosing the coefficients in the Nitsche numerical fluxes locally on each element and letting them depend on the relative area/volume of each side of the interface the unfitted finite element methods in~\cite{HaHa02,BBH09} can allow for discontinuities internal to the elements with optimal convergence order. However, these methods suffer from ill-conditioning just as XFEM. In~\cite{Bu10} and later in~\cite{BH12,BH11} a stabilization of the classical Nitsche's method for the imposition of Dirichlet boundary conditions on a boundary not fitted to the mesh was considered for the Poisson problem and for the Stokes equation. In these methods the stabilization is applied in the boundary region and optimal convergence order and well conditioned system matrices are ensured. For the elliptic interface problem other stabilization strategies have been suggested as remedies to the ill-conditioning problem, see e.g.~\cite{ZuCaCo11,WZKB}. We also refer to \cite{MaLa13}  for implementation aspects in three dimensions and \cite{JoLa13} for extensions to higher order elements.

{In this paper, we propose an accurate and stable finite element method for a Stokes interface problem involving two immiscible fluids with different viscosities and surface tension. }The model consists of the incompressible Stokes equations in two subdomains, each occupied by a fluid. Differences in viscosity between the fluids and the surface tension force poses jump conditions at the interface separating the fluids. {Our finite element method} enforces the jump conditions at the interface weakly with weighted coefficients in the Nitsche numerical fluxes. We also suggest slight changes to the variational formulation in~\cite{BBH09} to reduce spurious velocity oscillations and we include stabilization terms both for the velocity and the pressure that guarantee a well conditioned system matrix. The stabilization terms are consistent least squares 
terms controlling the jump in the normal gradient across faces between elements in a neighbourhood of the interface.  Using the stabilization terms we prove an inf-sup condition and that the resulting stiffness matrix has optimal conditioning. We also prove the inf-sup stability of the method under the condition that the mean value of the pressure in the entire domain is fixed, in contrast to the inf-sup result in~\cite{BBH09} which is based on the more restrictive condition that the mean values in each of the two subdomains are fixed. Our inf-sup result is also uniform with respect to the jump in the viscosity. The proposed method is simple to implement as it uses standard continuous linear basis functions with changes only in the variational form. 

The outline of this paper is as follows. In Section 2 we formulate the Stokes system and the finite element method. In Section 3 we prove that the method is of optimal convergence order. In Section 4 we prove that the condition number is $\mathcal{O}(h^{-2})$ independent of the position of the interface relative to the mesh. Finally, in Section~\ref{sec:NumEx}, we show numerical examples in two space dimensions and compare the method to existing techniques. We summarize our results in Section~\ref{sec:conc}.

\section{The interface problem and the finite element method}
We consider a problem consisting of two immiscible fluids separated by an interface, with the flow described by the incompressible Stokes equations. The Stokes system is a standard model for creeping viscous flow. In this section we present the equations and a finite element method for their approximate solution. 

\subsection{The two-fluid incompressible Stokes equations}
Let $\Omega$ be an {open bounded domain} in $\mathbb{R}^2$, with convex polygonal boundary $\partial \Omega$. We 
assume that two immiscible incompressible fluids occupy subdomains $\Omega_i \subset \Omega$, $i=1,2$ 
such that $\overline{\Omega} = \overline{\Omega}_1 \cup \overline{\Omega}_2$ and $\Omega_1 \cap \Omega_2 = \emptyset$ and that a smooth interface defined by $\Gamma = \partial \Omega_1 \cap \partial \Omega_2$ separates the immiscible fluids. 

We consider the following Stokes interface boundary value problem modeling two fluids with different viscosity and with surface tension: find the velocity $\bfu: \Omega
\rightarrow \mathbb{R}^2$ and the pressure $p:\Omega \rightarrow \mathbb{R}$ such that
\begin{subequations}\label{strongform}
\begin{alignat}{3}
- \nabla \cdot ( \mu \bfeps ( \bfu ) - p \bfI ) &= \bff & \quad & \text{in $\Omega_1 \cup \Omega_2$},\label{eq:diffeq}
\\
\nabla \cdot \bfu  &=0 &\quad & \text{on $\Omega_1\cup \Omega_2$},\label{eq:divu}
\\
\ldb \bfu \rdb &= \bfzero & \quad & \text{on $\Gamma$}, \label{eq:jumpu}
\\
\ldb ( \mu \bfeps ( \bfu ) - p \bfI )\bfn   \rdb  &= \sigma \kappa \bfn & \quad & \text{on $\Gamma$}, 
\label{eq:jumpnormalS}
\\
\bfu &=\bfg & \quad & \text{on $\partial \Omega$}.\label{eq:uBC}
\end{alignat}
\end{subequations} 
Here $\bfeps(\bfu) = (\nabla \bfu + (\nabla \bfu)^T)/2$ is the strain {rate} tensor, 
$\mu = 2\mu_i>0$ on $\Omega_i, i=1,2$ is a piecewise constant viscosity 
function on the partition $\Omega_1 \cup \Omega_2$ {(note that this definition of  $\mu$ yields $ \mu \bfeps ( \bfu )=2\mu_i  \bfeps ( \bfu )$)}, $\bff \in [L^2(\Omega)]^2$ 
and $\bfg \in [H^{1/2}(\partial \Omega)]^2$ are given functions, $\sigma$ is the surface 
tension coefficient, $\kappa$ is the curvature of the interface, $\bfn$ is the unit normal to $\Gamma$, outward-directed with respect to $\Omega_1$, and $\ldb a\rdb = (a_1 - a_2)|_\Gamma$ {is the jump}, 
where $a_i = a|_{\Omega_i}, i=1,2.$ 

We assume a constant surface tension coefficient $\sigma$. Hence, from equation~\eqref{eq:jumpnormalS}, it follows that across the interface the shear stress is continuous, i.e., 
\begin{equation}
\ldb \mu \bfeps(\bfu)  \bfn  \rdb  \cdot \bft =0. \label{eq:shearstress}
\end{equation} 
Hereinafter, we also denote the outward directed unit normal vector on $\partial \Omega$ by $\bfn$. We assume global conservation of mass, that is
\begin{equation}\label{eq:condonuatBC}
\int_{\partial \Omega} \bfg \cdot \bfn ds =0.
\end{equation}

We will use the notation $(\cdot,\cdot)_\omega$ for the $L^2(\omega)$ inner product on $\omega$ (and similarly for inner products in $[L^2(\omega)]^2$ and $[L^2(\omega)]^{2 \times 2}$). We let $\| v \|_{s,\omega}$ and $|v|_{s,\omega}$ denote the Sobolev norms and seminorms associated with the spaces $H^s(\omega)$, respectively.   

{We will in particular consider viscosity parameters $\mu_1$ 
and $\mu_2$ that satisfy the following assumption 
\begin{equation}\label{eq:assumptionmu}
c \leq \mu_1 \leq C,\qquad  0 < \mu_2 \leq C
\end{equation}
in other words both $\mu_1$ and $\mu_2$ are bounded, $\mu_1$ can not approach zero, 
while $\mu_2$ is positive but can be arbitrarily small. Under this assumption the constants 
in our stability and error estimates are independent of $\mu_1$ and $\mu_2$. 
Introduce the following spaces and corresponding norms
\begin{equation}
M=\{p \in L^2(\Omega) : (\mu^{-1}p,1)_{\Omega}
=0 \}, \qquad \| p \|^2_M = ( \mu^{-1} p,p)_{\Omega},
\end{equation}
and 
\begin{equation}
{\bfV}_{\bfg} = \{\bfv \in [H^1(\Omega)]^2 : \text{$\bfv = \bfg$ on $\partial \Omega$}\}, \qquad \|\bfv \|_{\bfV}^2 = (\mu \bfeps(\bfv), \bfeps(\bfv))_{\Omega}.
\end{equation}
The weak form of~\eqref{strongform} is: given $\bff \in \bfV'_0$ 
find $(\bfu,p) \in \bfV_\bfg \times M$ such that
\begin{equation}\label{eq:weakf}
( \mu \bfeps( \bfu), \bfeps( \bfv))_{\Omega} 
-(\nabla \cdot \bfv, p )_{\Omega}
+(\nabla \cdot \bfu, q )_{\Omega}
=(\bff,\bfv)_\Omega  + (\sigma \kappa ,  \bfv \cdot \bfn)_{\Gamma} \quad \forall (\bfv,q) \in \bfV_0 \times M, 
\end{equation}
and $\bfV'_0$ denotes the dual of $\bfV_0$. Note that for $\Gamma$ sufficiently 
smooth we have $\sup_{\bfx \in \Gamma} | \kappa(\bfx) |<c<\infty$ and hence  
using a trace inequality
\begin{equation}
|(\sigma \kappa ,  \bfv \cdot \bfn)_{\Gamma}| 
\leq \sigma C\|\bfv \|_{0,\Gamma} 
\leq \sigma C\|\bfv  \|_{1,\Omega_1}
\leq \sigma C\| \bfv \|_{\bfV},
\quad \forall \bfv \in \bfV_0, 
\end{equation} 
where we used (\ref{eq:assumptionmu}) in the last inequality. Note that for $ \bfv \in \bfV_0$, $\| \bfv \|_{\bfV}$ is equivalent with the $H^1$-norm by Korn's inequality. Thus, the right hand 
side of equation~(\ref{eq:weakf}) is well-defined. 
Next we have 
\begin{equation}
|(\mu \bfeps(\bfu),\bfeps(\bfv))| \leq \| \bfu \|_{\bfV} \| \bfv \|_{\bfV},
\end{equation}
\begin{equation}
|(\nabla \cdot \bfv, p)| 
\leq \| \mu^{1/2} \nabla \cdot \bfv \|_\Omega \| \mu^{-1/2} p \|_\Omega
\leq \sqrt{2} \| \bfv \|_{\bfV} \| p \|_M,
\end{equation} 
and the coercivity condition
\begin{equation}
|(\mu \bfeps(\bfv),\bfeps(\bfv))| \geq \| \bfv \|_{\bfV}^2.
\end{equation}
Furthermore,  the inf-sup condition 
\begin{equation}
C \| p \|_{M} \leq \sup_{\bfv \in \bfV_0} \frac{(\nabla \cdot \bfv,p)_\Omega}{\|\bfv \|_{\bfV}}
\end{equation}
holds with constant $C>0$ independent of $\mu_1$, $\mu_2$ under assumption (\ref{eq:assumptionmu}), see 
Theorem 2.2 in \cite{OR06}. Thus, problem~\eqref{eq:weakf} is well-posed and there exists a unique solution in $\bfV_g \times M$, cf.~\cite{BrFo91}.
}

For the convergence analysis we assume that the pressure space is
\begin{equation}
\mcV=\{p \in H^1(\Omega_1 \cup \Omega_2) : (\mu^{-1}p,1)_{\Omega_1\cup \Omega_2}=0 \}
\end{equation}
and the velocity space is
\begin{equation}
\mcW=\{\bfv \in  [H^2(\Omega_1 \cup \Omega_2)]^2 : \text{$\bfv = \bfg$ on $\partial \Omega$}\}.
\end{equation}

\subsection{Mesh and assumptions}\label{subsect:mesh}
Let $\mcK_h$ be a triangulation of $\Omega$, generated independently of the location of the interface $\Gamma$. Introduce the set of all element faces $\mcF$ associated with the mesh $\mcK_h$, the set of all elements that intersect the interface 
\begin{equation}\label{eq:interfset}
\mcK_\Gamma = \{ K \in \mcK_h : |\Gamma \cap \overline{K}|>0 \}, 
\end{equation}
and the set of all elements on the boundary 
\begin{equation}
\mcK_{\partial \Omega} = \{ K \in \mcK_h : |\partial \Omega \cap \overline{K} |>0 \}.   
\end{equation}
Define meshes on the subdomains $\Omega_i, i = 1, 2,$ as follows
{
\begin{equation}
\mcK_{h,i} = \{ K \in \mcK_h:  |\overline{\Omega}_i \cap \partial K|>0\}
\end{equation}
}
and let
\begin{equation}
\omegahi = \bigcup_{K \in \mcK_{h,i}} K, \quad \tilde{\omega}_{h,i} = \bigcup_{K \in \mcK_{h,i}, K\subset \Omega_i} K, \quad i=1,2.
\end{equation} 
Note that the interface $\Gamma$ is allowed to intersect the elements in $\mcK_h$ and $\mcK_{h,i}$ and that elements intersected by the interface are both in $\Omega_{h,1}$ and  $\Omega_{h,2}$. 

Let $\tilde{\mcK}_\Gamma$ be the set of all elements that share two of its faces with elements in $\mcK_\Gamma$, let 
\begin{equation}\label{eq:stabsets}
\mcF_{\Gamma,i} = \{ F \in \mcF : F \subset \partial K, K \in
\mcK_\Gamma \cup  \tilde{\mcK}_{\Gamma}, F \cap \overline{\Omega}_i \neq \emptyset, F\cap\partial
\Omega= \emptyset \}
\end{equation}
be the set of faces of elements in $\mcK_\Gamma \cup  \tilde{\mcK}_{\Gamma}$, that have a nonempty intersection with $\Omega_i$, and are not on the boundary $\partial \Omega$, and let
\begin{equation}
\omegain =  \tilde{\omega}_{h,i}  \setminus \bigcup_{K \in \tilde{\mcK}_{\Gamma}} K, \quad i=1,2,
\end{equation}
contain the elements in $\tilde{\omega}_{h,i}$ that are not in $\tilde{\mcK}_\Gamma$. 
{We modify 
the set $\tilde{\omega}_{h,i}$ in order to ensure that there is no element in 
$\omegain$ with two edges on the boundary $\partial \omegain$. This is a necessary 
condition for the inf-sup condition to hold on $\omegain$, see Lemma~\ref{lem:infsupbpi}. }
See  Fig.~\ref{fig:illustedg} for an illustration of the sets $\omegahi$, $\omegain$, and $\mcF_{\Gamma,i} $, i=1,2 in a two-dimensional case.
\begin{figure}
\centering
\includegraphics[scale=0.5]{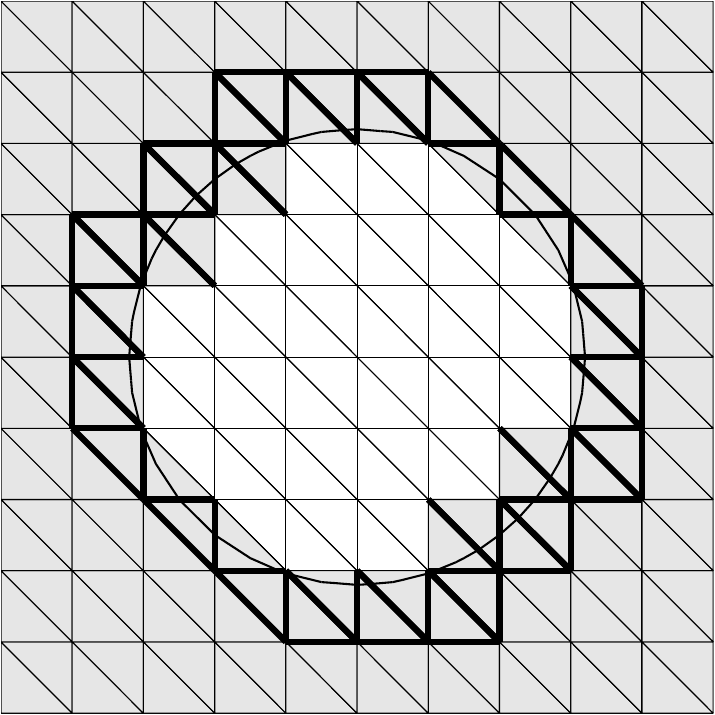} \hspace{1 cm}
\includegraphics[scale=0.5]{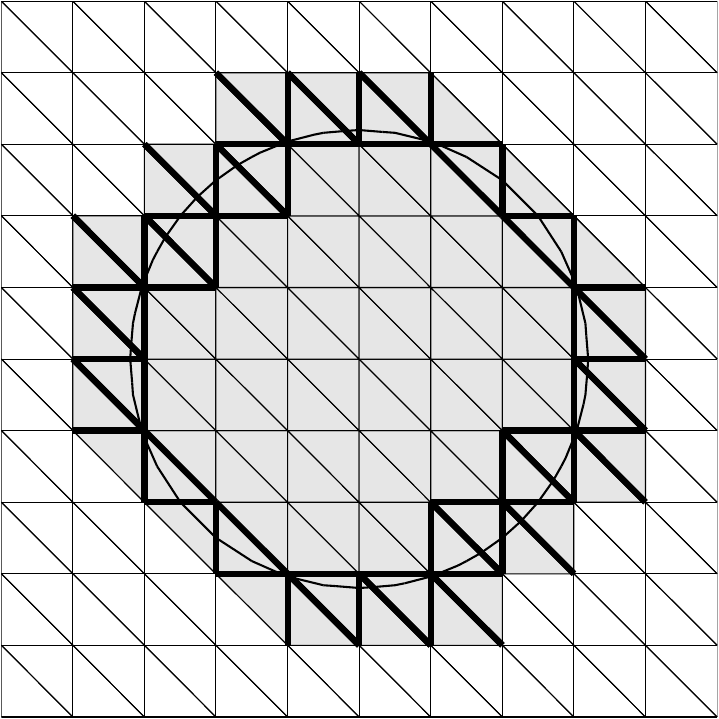}
\caption{A triangulation $\mcK_h$ of  the domain $\Omega$. The interface $\Gamma$ is a circle separating the subdomains $\Omega_1$ and $\Omega_2$.  Left panel: $\Omega_{h,1}$ (shaded triangles),  $\omega_{h,2}$ (non shaded triangles), and edges in $\mcF_{\Gamma,1}$ (thick lines).  Right panel: $\Omega_{h,2}$ (shaded triangles),  $\omega_{h,1}$ (non shaded triangles), and edges in $\mcF_{\Gamma,2}$ (thick lines).  \label{fig:illustedg}}
\end{figure}

We make the following assumptions:
\begin{itemize}\label{a:mesh}
\item \textbf{Assumption 1}: We assume that the triangulation is quasi-uniform, i.e., there exists {positive} constants $c_1$ and $c_2$ such that
\begin{equation}
c_1h \leq h_K\leq c_2h \quad \forall K \in \mcK_h,
\end{equation}
where $h_K$ is the diameter of $K$ and $h=\max_K h_K$, and that there are no elements 
with two edges on the boundary $\partial \Omega$.
\item \textbf{Assumption 2}: We assume that $\Gamma$ either intersects the boundary
$\partial K$ of an element $K\in \mcK_\Gamma$ exactly twice and each
(open) edge at most once, or that $\Gamma \cap \overline{K}$ coincides with an edge of the element.
\item \textbf{Assumption 3}:
Let $\Gamma_{K,h}$ be the straight line segment connecting the points of intersection between $\Gamma$ and $\partial K$. We assume that $\Gamma_K=\Gamma \cap K$ is a function of length on $\Gamma_{K,h}$; in local coordinates:
\begin{equation}
\Gamma_{K,h}=\{ (\xi, \eta): 0<\xi<|\Gamma_{k,h}|, \eta=0 \}
\end{equation}
and 
\begin{equation}
\Gamma_{K}=\{ (\xi, \eta): 0<\xi<|\Gamma_{k,h}|, \eta=\delta(\xi) \}.
\end{equation}
\item \textbf{Assumption 4}: We assume that for each $K\in \mcK_\Gamma$ there are elements $K^i \subset \Omega_i$, $i=1, 2$ such that $\overline{K}\cap \overline{K^i}\neq \emptyset$.
\item \textbf{Assumption 5}:  We assume that the mesh coincides with the outer boundary 
$\partial \Omega$. 
\end{itemize}
Assumptions 2-3 essentially state that the interface is well resolved by the mesh. Except that we in the second
assumption also allow the interface to be aligned with a mesh line these two assumptions are as in~\cite{HaHa02}.  Assumption 4 states that each $K\in \mcK_\Gamma$ shares a face or at least a vertex with an element $K^1\subset \Omega_1$ and an element $K^2\subset \Omega_2$. For each $\Omega_i$ this is the same assumption as in~\cite{BH11}.

\subsection{The finite element method}\label{sec:FEM}
Let $\mcK_h$  be a triangulation of $\Omega$ that satisfies all the assumptions in Section~\ref{subsect:mesh}.  
We let $\mcV_{h,\mu}$ be the space of continuous piecewise linear polynomials defined on $\mcK_h$ with $(\mu^{-1}q_h,1)_{\Omega_1\cup\Omega_2}=0$ $\forall q_h\in \mcV_{h,\mu}$
and we let
\begin{equation}
{\mcV_h = \{ p = (p_{h,1},p_{h,2}) :  p_{h,i} \in \mcV_{h,i}, i=1,2\}}
\end{equation}
be our pressure space where
\begin{equation}
\mcV_{h,i} = \mcV_{h,\mu} |_{\omegahi},\quad i = 1,2,
\end{equation}
i.e. the spaces of restrictions to $\omegahO$ and $\omegahT$ of
functions in $\mcV_{h,\mu}$. For an illustration in a one-dimensional model case, see Fig.~\ref{fig:illust}. 
Note that $p_h \in \mcV_{h}$ is double valued on elements in $\mcK_\Gamma$ and is allowed to be discontinuous at the interface $\Gamma$. We define $(p_h,q_h)_{\Omega_1\cup \Omega_2}=\sum_{i=1}^2 (p_{h,i},q_{h,i})_{\Omega_i}$.
\begin{figure}\centering
\includegraphics[scale=0.6]{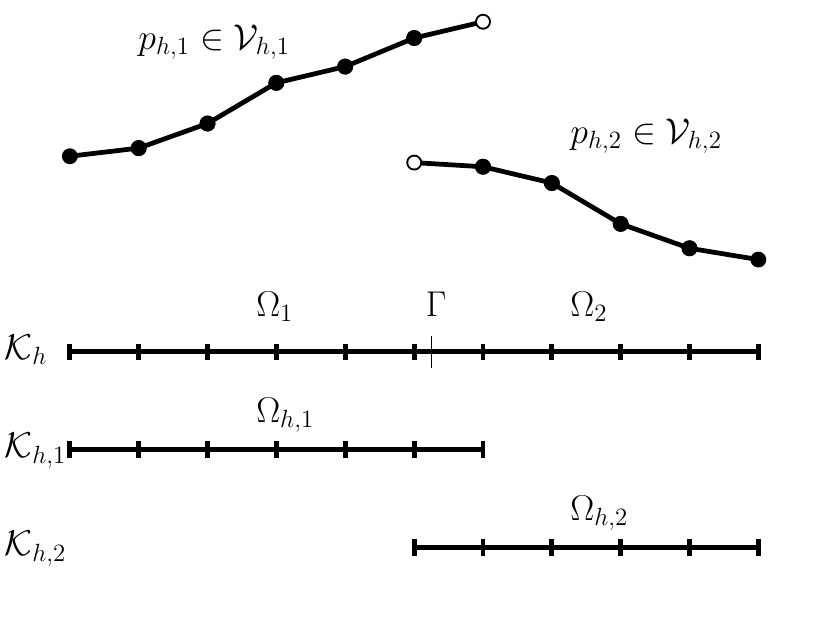}
\caption{Illustration of the domains, meshes, and spaces in a
  one-dimensional model case. \label{fig:illust}}
\end{figure}

In the same way as above we construct the velocity space but on a uniform refinement of $\mcK_h$, denoted by $\mcK_{h/2}$. The triangulation $\mcK_{h/2}$ also satisfies all the assumptions in Section~\ref{subsect:mesh} and we define all the parameters in Section~\ref{subsect: parameters} on the mesh $\mcK_{h/2}$. To construct the velocity space we let $\mcW_{h,0}$  be the space of vector valued continuous piecewise linear polynomials on 
$\mcK_{h/2}$. We will impose the Dirichlet conditions weakly. Thus, there are no special boundary restrictions at $\partial \Omega$ on the velocity space. We define  
\begin{equation}
\mcW_{h,i} = \mcW_{h,0}|_{\omegahi}, \quad i = 1,2,
\end{equation}
and we let 
\begin{equation}
{\mcW_h = \{\bfu = (\bfu_{h,1},\bfu_{h,2}) : \bfu_{h,i} \in \mcW_{h,i}, i=1,2\}}.
\end{equation}
We drop the subscript $h$ and $h/2$ and denote both the velocity and the pressure mesh by $\mcK$. An
element $K\in \mcK$ is an element in $\mcK_{h/2}$ whenever we have terms involving functions in the velocity space $\mcW_h$ and it is an element in $\mcK_{h}$ whenever we have terms involving only functions in the pressure space.

We propose the following Nitsche method: find $(\bfu_h, p_h) \in \mcW_h \times \mcV_h$ such that
\begin{equation}\label{eq:dg}
  A_h(\bfu_h, p_h ; \bfv_h, q_h )+ \varepsilon_{\bfu} J_{\bfu}(\bfu_h,\bfv_h)+\varepsilon_p J_p(p_h,q_h) = L_h(\bfv_h), \quad \forall (\bfv_h, q_h) \in \mcW_h \times \mcV_h.
\end{equation} 
Here $A_h(\cdot;\cdot)$ is a  bilinear form defined by
\begin{equation}
A_h(\bfw,r;\bfv,q) = a_h(\bfw,\bfv) +b_h(\bfw,q) - b_h(\bfv,r),
\end{equation}
where
\begin{align}
a_h(\bfw,\bfv) &= ( \mu \bfeps( \bfw), \bfeps( \bfv))_{\Omega_1 \cup \Omega_2}
 \nonumber \\
&\qquad - (\{ \mu  \bfeps(\bfw)  \bfn  \} , \ldb \bfv \rdb  )_{\Gamma} 
-  (\ldb \bfw \rdb , \{ \mu  \bfeps(\bfv)  \bfn  \})_{\Gamma} 
\nonumber \\
&\qquad - (\mu  \bfeps(\bfw)  \bfn  , \bfv)_{\partial \Omega} 
-  (\bfw, \mu  \bfeps(\bfv)  \bfn  )_{\partial \Omega} 
\nonumber \\ \label{nitschepenalty}
&\qquad  + \lambda_\Gamma (\ldb \bfw \rdb , \ldb \bfv \rdb )_{\Gamma} + \lambda_{\partial \Omega} ( \bfw, \bfv)_{\partial \Omega}
\\
b_h^1(\bfw,q) &= (\nabla \cdot \bfw, q )_{\Omega_1 \cup \Omega_2} - (\ldb  \bfn \cdot \bfw \rdb  ,  \{ q  \})_\Gamma  \label{eq:bdivu}
- (\bfn \cdot \bfw, q )_{\partial \Omega}\\
b_h^2(\bfw,q) &= -(\bfw,\nabla q )_{\Omega_1 \cup \Omega_2} + (\ldb  q \rdb  ,  \langle \bfw \cdot \bfn \rangle)_\Gamma,  \label{eq:bgradp}
\end{align}
and $L_h(\cdot)$ is a linear functional defined by
\begin{align}\label{eq:l}
L_h(\bfv) &= (\bff,\bfv)_\Omega  + (\sigma \kappa , \langle \bfv \cdot \bfn \rangle)_{\Gamma} + (\bfg, \lambda_{\partial \Omega} \bfv - \mu \bfeps(\bfv)\bfn +\bfn q)_{\partial \Omega}.  
\end{align}
{We have used the average operators
\begin{equation}\label{averageop}
\{ a \} = \kappa_1 a_1 + \kappa_2 a_2,
\qquad
\langle a \rangle  = \kappa_2 a_1 + \kappa_1 a_2,
\end{equation}
where the weights $\kappa_1$ and $\kappa_2$ are real numbers satisfying
$\kappa_1+\kappa_2 = 1$. We specify the precise choice of the weights $\kappa_1$ and $\kappa_1$ as well as the penalty parameters $\lambda_\Gamma$ and $\lambda_{\partial \Omega}$ in the following section. The two forms $b_h^1$ and $b_h^2$ in equation~\eqref{eq:bdivu} and~\eqref{eq:bgradp}, respectively  are mathematically equivalent. By integrating the bilinear form $b_h^1$  we get $b_h^2$. However, we recommend to use $b_h=b_h^2$  in simulations since it results in reduced spurious velocities, see Remark~\ref{rem:staticD}.}

In equation~\eqref{eq:dg} $\varepsilon_{\bfu}$ and $\varepsilon_p$ are positive constants and the stabilization terms are defined as
\begin{align}\label{eq:stabp}
J_p(p_h,q_h) &=  \sum_{i=1}^2  \sum_{F \in \mcF_{\Gamma,i}}  \mu_i^{-1}h^3 \left( \ldb \bfn_F \cdot \nabla p_{h,i} \rdb_F ,\ldb  \bfn_F \cdot \nabla q_{h,i} \rdb_F \right)_F
\end{align}
and the component wise extension for vector valued functions $\bfu_{h,i}=(\bfu^1_{h,i},\bfu_{h,i}^2)$ 
\begin{equation}\label{eq:stabu}
J_{\bfu}(\bfu_h,\bfv_h) = \sum_{j=1}^2\sum_{i=1}^2 \sum_{F \in \mcF_{\Gamma,i}} \mu_i h^s \left( \llbracket \bfn_F \cdot \nabla \bfu_{h,i}^j \rrbracket_F  , \llbracket  \bfn_F \cdot \nabla \bfv_{h,i}^j \rrbracket_F \right)_F.  
\end{equation}
We choose $s$ in different ways  depending on if the interface cuts the domain boundary or not. We take $s=1$ only for those faces that have to be crossed to pass from an element on the boundary $K \in \mcK_{\partial \Omega}$, $K\not \subset \Omega_i$, $K \cap \Omega_i\neq \emptyset$, to the closest element $K^i\subset \Omega_i$, otherwise $s=3$. We have employed the following notation for the jump in a function $v$ at an interior face $F$
\begin{equation}
\ldb v \rdb_F  = v^+ - v^-,
\end{equation}
where $v^\pm(\bfx) = \lim_{t \rightarrow 0^+} v(\bfx \mp t \bfn_F)$,
for $\bfx \in F$, and $\bfn_F$ is a fixed unit normal to $F$. 

\begin{rem}
 {The stabilization terms that appear in the method are all consistent and provide the control necessary to prove that the method satisfy the inf-sup condition and that the resulting algebraic system is well conditioned. More precisely, the stabilization terms in (\ref{nitschepenalty}) are the standard 
Nitsche, or interior penalty terms, that are used to ensure that the form $a_h(\cdot,\cdot)$ is coercive. 
In case the interface cuts the domain boundary we also need the stabilization term $J_{\bfu}(\bfu_h,\bfv_h)$ in equation~\eqref{eq:stabu} to ensure coercivity. The stabilization term $J_p(p_h,q_h)$ in equation~\eqref{eq:stabp} is used to prove the inf-sup stability of the method. Finally, to control the condition number of the system matrix independently of the position of the interface relative to the mesh both the stabilization terms $J_p(p_h,q_h)$ and $J_{\bfu}(\bfu_h,\bfv_h)$ are needed, see Fig.~\ref{fig:condvsdist}.} The sets $\mcF_{\Gamma,i} $, i=1,2  in equation~\eqref{eq:stabp} are defined for the pressure mesh $\mcK_{h}$  and the sets $\mcF_{\Gamma,i} $, i=1,2 in equation~\eqref{eq:stabu} are defined for the velocity mesh $\mcK_{h/2}$ and are different. The face $F$ is always a full face in the underlying mesh. Similar stabilizations were used in the fictitious domain methods of~\cite{BH11}. However, the set of faces that are stabilized are slightly different here and we choose $s$ in equation~\eqref{eq:stabu} in different ways  depending on if the interface cuts the domain boundary or not. 
\end{rem}

\begin{rem}
 The computation of the bilinear forms $a_h$ and $b_h$ require integration over $\Omega_i\subset \Omega_{h,i}$. Thus, for elements cut by the interface $\Gamma$ the integration should be performed only over parts of the elements. The functions $p_{h,i}$ and $\bfu_{h,i}$ are defined on the larger subdomains $\Omega_{h,i}$ and the stabilization terms~\eqref{eq:stabp} and~\eqref{eq:stabu} ensure well defined extensions from $\Omega_i$ to $\Omega_{h,i}$. 
\end{rem}

{
\begin{rem}
A nonsymmetric interior penalty method 
\begin{align}
a_h(\bfw,\bfv) &= ( \mu \bfeps( \bfw), \bfeps( \bfv))_{\Omega_1 \cup \Omega_2}
 \nonumber \\
&\qquad - (\{ \mu  \bfeps(\bfw)  \bfn  \} , \ldb \bfv \rdb  )_{\Gamma} 
+  (\ldb \bfw \rdb , \{ \mu  \bfeps(\bfv)  \bfn  \})_{\Gamma} 
\nonumber \\
&\qquad - (\mu  \bfeps(\bfw)  \bfn  , \bfv)_{\partial \Omega} 
+  (\bfw, \mu  \bfeps(\bfv)  \bfn  )_{\partial \Omega} 
\nonumber \\
&\qquad  + \lambda_\Gamma (\ldb \bfw \rdb , \ldb \bfv \rdb )_{\Gamma} + \lambda_{\partial \Omega} ( \bfw, \bfv)_{\partial \Omega}
\end{align}
may also be used. This approach leads to some simplifications 
in the proof of coercivity in Lemma \ref{lem:coera}, since we do not need to use an inverse 
inequality, see Lemma \ref{lem:inverseineq}, but on the other hand standard Nitsche duality 
arguments can not be used to prove $L^2$ estimates. It is however still necesseary to add the 
stabilization terms (\ref{eq:stabp}) and (\ref{eq:stabu}) to ensure that the resulting linear system 
of equations is well conditioned.
\end{rem}
}
 
{ \subsection{Penalty parameters and averaging operators} \label{subsect: parameters}
%
For the weights used in (\ref{averageop})
we consider the following two cases, in accordance with the intersection
options of $\Gamma$:
\begin{itemize}
\item 
For each element $K\in \mcK_\Gamma$ such that $\Gamma$ intersects
the boundary of the element  exactly twice, we have $|K\cap\Omega_i| =
\alpha_{i,K}^{\phantom2}h_K^2$ and $|\Gamma \cap K|=\gamma_K h_K$ for some $\alpha_{i,K}$, $\gamma_K$ $>0$, and we define
\begin{equation}\label{eq:kappa}
\kappa_1|_{ K} =
\frac{\mu_2\alpha_{1,K}}{\mu_1\alpha_{2,K}+\mu_2\alpha_{1,K}}, \quad
\kappa_2|_{ K} =
\frac{\mu_1\alpha_{2,K}}{\mu_1\alpha_{2,K}+\mu_2\alpha_{1,K}},
\end{equation}
and the penalty parameter
\begin{equation} \label{eq:lambda}
  \lambda_\Gamma|_{K} = \frac{\{ \mu \}}{h_K} \lp D + C\frac{\gamma_K}{\alpha_K}\rp, \quad  \alpha_K=\alpha_{1,K}+\alpha_{2,K},  \ C>1, D>0.
\end{equation}  
\item If $\Gamma \cap \overline{K}$ coincides with an element edge, then
$\Gamma \cap \overline{K}$ will also be an edge of another triangle
$T\in \mcK_\Gamma$ and $\Gamma \cap \overline{K}=\Gamma \cap
\overline{T}$.  We may without loss of generality assume that
$T\subset \Omega_1$ and $K\subset \Omega_2$.  We may write $|T| =
\alpha_T^{\phantom2}h_T^2$ and $|K| = \alpha_K^{\phantom2}h_K^2$ for
some $\alpha_T$, $\alpha_K>0$, and $|\Gamma \cap \overline{K}| =
|\Gamma \cap \overline{T}| = \gamma_K h_K = \gamma_T h_T$ for some
$\gamma_K$, $\gamma_T>0$.  We define
\begin{equation}\label{eq:kappaedg}
\begin{aligned}
 \kappa_1|_{K}=
 \frac{\mu_2\alpha_T/\gamma_T^2}{\mu_1\alpha_K/ \gamma_K^2+\mu_2\alpha_T/ \gamma_T^2}, 
 \quad 
 \kappa_2|_{K}=
 \frac{\mu_1\alpha_{K}/\gamma_K^2}{\mu_1\alpha_{K}/\gamma_K^2+\mu_2\alpha_{T}/\gamma_T^2},
\end{aligned}
\end{equation}
and the penalty parameter 
\begin{equation}\label{eq:lambdaedg}
\lambda_\Gamma|_{K} = \frac{\{ \mu \}}{h_K}\lp D+  \frac{C/\gamma_K}{\alpha_T/\gamma_T^ 2+\alpha_{K}/{\gamma_K^2}} \rp, \quad  C>1, D>0.
\end{equation}
Note that $0 \leq \kappa_i\leq 1$ and $\kappa_1+\kappa_2=1$ in both
cases.
For elements $K$ in $\mcK_{\partial \Omega}$ we also define the penalty parameter
\begin{equation}\label{eq:lambdaB}
\lambda_{\partial \Omega}|_{K \cap \Omega_i}=\frac{\mu_i}{h_K}\left(G+H\frac{\gamma_{\partial \Omega, K}}{\alpha_{K}}\right), \quad  G>0,  \textrm{ $H$ sufficiently large}.
\end{equation} 
\end{itemize}
}
{
\begin{rem} \label{remmu}
Under the assumption (\ref{eq:assumptionmu}) we consider the case when $\mu_1$ is constant 
and $\mu_2\rightarrow 0^+$. Then we have 
\begin{align}\label{eq:kappa}
\kappa_1|_{ K} &=
\frac{\mu_2\alpha_{1,K}}{\mu_1\alpha_{2,K}+\mu_2\alpha_{1,K}}
\rightarrow 0^+
\\
\kappa_2|_{ K} &=
\frac{\mu_1 \alpha_{2,K}}{\mu_1\alpha_{2,K}+\mu_2\alpha_{1,K}}
\rightarrow 1
\end{align}
Furthermore, we note that 
\begin{equation}
\{ \mu \}= \kappa_1 \mu_1 + \kappa_2 \mu_2 \rightarrow 0^+
\end{equation}
since $\kappa_1 \rightarrow 0^+$ and $\mu_2 \rightarrow 0^+$, 
and thus
\begin{equation} \label{eq:lambda}
  \lambda_\Gamma|_{K} = \frac{\{ \mu \}}{h_K} \lp D + C\frac{\gamma_K}{\alpha_K}\rp \rightarrow 0^+
\end{equation} 
These results show that the interface condition in this case converges to a Neumann condition since 
all the interface terms vanish in the limit. 
\end{rem}
}

\section{Analysis}
In this section we will show that the finite element method presented in Section~\ref{sec:FEM} has optimal convergence order. {Throughout this section all constants are positive and independent of the mesh size and we use $b_h=b_h^1(\cdot,\cdot)$.} We begin by proving the following consistency relation for the finite element formulation~\eqref{eq:dg}.
\begin{lemma} \label{l:Galorth}
Let $(\bfu,p) \in \mcW\times \mcV$ be the solution to the boundary value problem~\eqref{strongform} and $(\bfu_h,p_h)$ be the solution of the finite element formulation~\eqref{eq:dg}. Then
\begin{equation}
  A_h(\bfu-\bfu_h, p-p_h ; \bfv_h, q_h ) = \varepsilon_{\bfu} J_{\bfu}(\bfu_h,\bfv_h) +\varepsilon_p J_p(p_h,q_h), \quad \forall (\bfv_h, q_h) \in \mcW_h \times \mcV_h.
\end{equation}
\end{lemma}
\begin{proof}
First, note that since $\kappa_1+\kappa_2=1$ we have
\begin{equation}
\ldb a b \rdb =\{a\} \ldb b\rdb+ \ldb a \rdb \langle b \rangle,
 \end{equation}
hence we can write 
\begin{equation}\label{eq:interfaceterm}
\int_\Gamma \ldb ((\mu \bfeps( \bfu )-p\bfI)\bfn)\bfv \rdb ds= \int_\Gamma \{  (\mu \bfeps( \bfu )-p\bfI)\bfn  \} \ldb \bfv \rdb ds + \int_\Gamma \ldb (\mu \bfeps( \bfu )-p\bfI) \bfn \rdb \langle \bfv \rangle ds.
\end{equation}
Using the interface conditions for the normal stress and the shear stress, equation~\eqref{eq:jumpnormalS} and~\eqref{eq:shearstress}, we have that   
\begin{equation} \label{eq:intjumpterm}
 \int_\Gamma \ldb  (\mu \bfeps( \bfu )-p\bfI) \bfn \rdb \langle \bfv \rangle ds=\int_{\Gamma} \sigma \kappa  \langle \bfv \cdot \bfn \rangle ds.
\end{equation}
Now, multiplying~\eqref{strongform} by a test function $(\bfv_h, q_h) \in \mcW_h \times \mcV_h$ and integrating by parts, using~\eqref{eq:interfaceterm} and~\eqref{eq:intjumpterm}, the boundary conditions~\eqref{eq:uBC} and~\eqref{eq:condonuatBC}, and that $\bfu$ is continuous~\eqref{eq:jumpu} we get
\begin{equation}
a_h(\bfu,\bfv_h) + b_h(\bfu,q_h) - b_h(\bfv_h,p) = L_h(\bfv_h),
\end{equation}
and the claim follows. 
\end{proof}

We introduce the following mesh dependent norms 
\begin{align}
\tn \bfv \tn^2 &= \|\mu^ {1/2} \bfeps(\bfv) \|^2_{0,\Omega_1 \cup \Omega_2} 
+ \| \{ \bar{\mu}^{1/2}  \bfeps(\bfv)  \bfn   \} \|^2_{-1/2,h,\Gamma} + \|\{ \mu\}^{1/2} [\bfv ]\|^2_{1/2,h,\Gamma} \nonumber \\ 
& \quad 
 +\|\mu^{1/2}  \bfeps(\bfv)  \bfn  \|^2_{-1/2,h,\partial \Omega} + \|  \mu^{1/2} \bfv \|^2_{1/2,h,\partial  \Omega}  \quad \forall  \bfv \in \mcW +\mcW_h,
\label{eq:tnormv} \\
\tn \bfv_h \tn^2_h &= \tn \bfv_h \tn^2 +  J_{\bfu}(\bfv_h,\bfv_h) \quad \forall  \bfv_h \in \mcW_h,
\label{eq:tnormvh} \\
\tn (\bfv,q) \tn^2 &= \tn \bfv \tn^2 + \|\mu^{-1/2} q\|_{0,\Omega_1 \cup \Omega_2}^2+\| \{ \mu\}^{-1/2} \{ q \} \|^2_{-1/2,h,\Gamma} 
 \nonumber \\
&\quad 
+ \|  \mu^{-1/2}q \|^2_{-1/2,h,\partial \Omega} 
\quad \forall (\bfv,q) \in (\mcW + \mcW_h) \times (\mcV+\mcV_h),
\label{eq:tnormup} \\
\tn (\bfv_h,q_h) \tn^2_h &= \tn \bfv_h \tn^2_h + \|\mu^{-1/2} q_h\|_{0,\omegahO\cup \omegahT}^2 + J_p(q_h,q_h) \quad \forall (\bfv_h,q_h) \in \mcW_h \times \mcV_h,\label{eq:tnormvqh}
\end{align}
where $\bar{\mu}_i^{1/2}=\mu_i \{\mu\}^{-1/2}$, $i=1,2$. The norms on the trace of a function on $\Gamma$ are defined by
\begin{align}
&\| \bfv \|_{1/2,h,\Gamma}^2 = \sum_{K\in \hat{\mcK}_\Gamma} h_K^{-1} \| \bfv \|_{0,\Gamma \cap \overline{K}}^2,
\label{eq:halfnormgamma} \\
&  \| \bfv \|_{-1/2,h,\Gamma}^2 = \sum_{K\in \hat{\mcK}_\Gamma}  h_K \| \bfv \|_{0,\Gamma \cap \overline{K}}^2,
\label{eq:minushalfnormgamma}
\end{align}
and similarly for the trace of a function on $\partial \Omega$. Here $\hat{\mcK}_\Gamma$ is the set of elements that intersect the interface but when a part of $\Gamma$ coincides with an element edge only one of the two elements sharing that edge belongs to $\hat{\mcK}_\Gamma$.
Note that 
\begin{align}
 &(\bfw,\bfv)_\Gamma \leq \| \bfw \|_{1/2,h,\Gamma} \| \bfv \|_{-1/2,h,\Gamma}, \label{eq:CSongamma} \\
 &(\bfw,\bfv)_{\partial \Omega} \leq \| \bfw \|_{1/2,h, \partial \Omega} \| \bfv \|_{-1/2,h,\partial \Omega}.\label{eq:CSonboundary}
\end{align}
We will need the following inverse inequality when proving the inf-sup stability of the finite element method. 
\begin{lemma} \label{lem:inverseineq}
Assume that $K$ is an element in $\mcK_\Gamma$ such that, for $i=1$ or $2$, $|K \cap \Omega_i|=\alpha_{i,K}h_K^2$, where $\alpha_{i,K}>0$ and let $|\Gamma \cap \overline{K}|=\gamma_K h_K$. 
For any function $\bfv \in \mcW_h$, the following inverse inequality holds 
\begin{equation}
  h_K \| \kappa_i \bfeps( \bfv )  \bfn \|^2_{0,\Gamma \cap \overline{K}} \le
  \frac{\kappa_i^2\gamma_K}{\alpha_{i,K}} \| \bfeps( \bfv ) \|^2_{0,K \cap \Omega_i}.
\end{equation}
\end{lemma}
Lemma~\ref{lem:inverseineq} follows using that the functions in $\mcW_h$ are linear, cf.~\cite{HGamm}. 

We also state two trace inequalities that we need for proving an approximation result. 
\begin{lemma} \label{lem:traceineq}
Let $K\in \mcK$ and $v \in H^1(K)$. There exists {positive} constants $C$ and $\tilde{C}$ such that for $s \in \mathbb{R}$ 
\begin{align}
h_K^{-s}\| v \|_{0,\Gamma \cap K}^2 &\leq C( h_K^{-1-s}\| v \|_{0,K}^2 + h_K^{1-s} \| v \|^2_{1,K} ),  \nonumber \\
h_K^{-s}\| v \|_{0,\partial K}^2 &\leq \tilde{C} (h_K^{-1-s}\| v \|_{0,K}^2 + h_K^{1-s} \| v \|^2_{1,K} ).  
\end{align}
\end{lemma}
Under Assumption 1-3 the first trace inequality follows from Lemma 3 in~\cite{HaHa02} and a scaling argument. The second trace inequality follows from a standard trace estimate, see~\cite[Theorem~1.6.6]{BrSc}. 

We will also need the following estimates:
\begin{lemma} \label{lem:invandtraceineq}
{Let $K\in \mcK$. There exist positive constants such that for all $v \in \mcW_h$ or $v \in \mcV_h$, we have} 
\begin{align}
\| \nabla v \|_{0,K}^2 &\leq Ch_K^{-2}\| v\|_{0,K}^2 \label{eq:inversineq}, \\
\| v\|_{0,\partial K}^2&\leq C h_K^{-1}\| v\|_{0,K}^2 \label{eq:traceineqF}, \\
\| v\|_{0,\Gamma\cap K}^2&\leq C h_K^{-1}\| v\|_{0,K}^2 \label{eq:traceineq}.
\end{align}
\end{lemma}
The inverse inequality~\eqref{eq:inversineq} follows from~\cite[Lemma 4.5.3]{BrSc} and the trace inequalities follow from Lemma~\ref{lem:traceineq} and the inverse inequality~\eqref{eq:inversineq}.

\subsection{Continuity}
We begin by showing the continuity of $a_h(\cdot,\cdot)$ and then we prove the continuity of $A_h(\cdot,\cdot; \cdot, \cdot)$. 
\begin{lemma} \label{lem:conta}
 Let $\bfu \in \mcW+\mcW_h$. Then,
\begin{equation}
a_h(\bfu,\bfv_h) \leq C_{cont} \tn \bfu \tn \; \tn \bfv_h \tn \quad \forall \bfv_h \in \mcW_h
\end{equation}
{where $C_{cont}$ is a positive constant independent of $\mu_1$ and $\mu_2$ under assumption (\ref{eq:assumptionmu})}.
\end{lemma}
\begin{proof} 
For each $K\in {\mcK}_\Gamma$ let $| \Gamma \cap \overline{K}|=\gamma_K h_K$ and $|K|=\alpha_Kh_k$ for some $\gamma_K, \alpha_K>0$. The claim follows by recalling the definition of $a_h(\cdot,\cdot)$, applying the Cauchy-Schwarz inequality, inequalities~\eqref{eq:CSongamma} and~\eqref{eq:CSonboundary} to the interface and boundary terms, and noting that
$\{ \mu \bfeps(\bfu)  \bfn  \}= \{ \bar{\mu}^{1/2}  \bfeps(\bfu)  \bfn  \} \{ \mu\}^{1/2},$ 
\begin{align}
\lambda |_{\Gamma\cap \overline{K}} &\leq \left( C \max_{K\in K_{\Gamma}} \left( \frac{\gamma_K}{\alpha_K} \right)+D\right) \frac{\{ \mu \}}{h_K},
\\
\lambda |_{\partial \Omega\cap \overline{K}} &\leq \left(H\max_{K\in K_{\partial \Omega}} \left( \frac{\gamma_{\partial \Omega,K}}{\alpha_K} \right) +G \right) \frac{\mu}{h_K}.
\end{align}
\end{proof}

\begin{lemma} \label{lem:contb}
Let $\bfu \in \mcW+\mcW_h$, $\bfv_h \in \mcW_h$, $p\in \mcV+ \mcV_h$, $q\in \mcV_h$. Then,
\begin{align}
A_h(\bfu, p; \bfv_h,q_h ) \leq C_{A} \tn (\bfu,p) \tn \  \tn (\bfv_h,q_h) \tn_h
\end{align}
{where $C_A$ is a positive constant independent of $\mu_1$ and $\mu_2$ under assumption (\ref{eq:assumptionmu})}.
\end{lemma}
\begin{proof} Starting from the definition of $A_h(\bfu , p; \bfv_h,q_h )$ we have
\begin{align}
A_h(\bfu , p; \bfv_h,q_h )& = a_h ( \bfu , v_h)  + b_h( \bfu, q_h)+ b_h( \bfv_h, -p)
\nonumber \\
&=I + II + III 
\end{align}
\noindent{\bf Term $\bfI$.} Using the continuity property of $a_h(\cdot,\cdot)$  (Lemma~\ref{lem:conta}), we have 
\begin{equation}
I \leq C_{cont}\tn \bfu \tn \, \tn \bfv_h \tn. 
\end{equation}
\noindent{\bf Term $\bfI\bfI$.} Using the Cauchy-Schwarz inequality and the trace inequalities in Lemma~\ref{lem:invandtraceineq} to bound the contributions from the interface and the boundary we get
\begin{align}
II &\leq \| \mu^{1/2}\nabla \cdot \bfu \|_{0,\Omega_1 \cup \Omega_2} \|\mu^{-1/2} q_h\|_{0,\Omega_1 \cup \Omega_2}
+\| \{ \mu \}^{1/2}  \ldb \bfu  \rdb \|_{1/2,h,\Gamma}  \| \{\mu\}^{-1/2} \{ q_h \} \|_{-1/2,h,\Gamma}
\nonumber \\
&+ \| \mu^{1/2} \bfu  \|_{1/2,h,\partial \Omega}
\| \mu^{-1/2}q_h  \|_{-1/2,h,\partial \Omega}
\nonumber \\
&\leq C \tn \bfu \tn \left( \| \mu^{-1/2}q_h\|_{0,\Omega_1 \cup \Omega_2} + \| \{\mu\}^{-1/2} \{ q_h \} \|_{-1/2,h,\Gamma}+\| \mu^{-1/2} q_h  \|_{-1/2,h,\partial \Omega} \right) 
\nonumber \\
&\leq C  \tn \bfu \tn  \| \mu^{-1/2} q_h \|_{0,\omegahO \cup \omegahT}. 
\end{align}
{
Here we used the estimate
\begin{align}
\{\mu\}^{-1} |\{ q_h \}|^2 &\leq C\left( \frac{\kappa_1^2}{\kappa_1 \mu_1 + \kappa_2 \mu_2} |q_{h,1}|^2 
+  \frac{\kappa_2^2}{\kappa_1 \mu_1 + \kappa_2 \mu_2} |q_{h,2}|^2 \right)
\nonumber \\
&= C\left( \kappa_1 \left(\frac{\kappa_1 \mu_1}{\kappa_1 \mu_1 + \kappa_2 \mu_2}\right) \mu_1^{-1}|q_{h,1}|^2 
+  \kappa_2 \left( \frac{\kappa_2 \mu_2}{\kappa_1 \mu_1 + \kappa_2 \mu_2}\right) \mu_2^{-1} |q_{h,2}|^2 \right)
\nonumber \\
&\leq C\left( \mu_1^{-1}|q_{h,1}|^2 + \mu_2^{-1} |q_{h,2}|^2 \right),
\end{align}
which holds pointwise on $\Gamma$.}

\noindent{\bf Term $\bfI\bfI \bfI$.}  Using Cauchy-Schwarz we obtain
\begin{align}
III&\leq 
C \tn \bfv_h \tn \left( \|\mu^{-1/2} p\|_{0,\Omega_1 \cup \Omega_2}^2 +  \| \{\mu\}^{-1/2} \{ p \} \|_{-1/2,h,\Gamma}+\| \mu^{-1/2} p  \|_{-1/2,h,\partial \Omega}^2 \right)^{1/2}. 
\end{align}
Summing the estimates of terms $ I-III$ and using the definition of the norms $\tn (\cdot, \cdot) \tn$ and  $\tn (\cdot,\cdot) \tn_h$ yields the claim. 
\end{proof}

\subsection{Inf-sup stability}
In this section we will show that the finite element formulation~\eqref{eq:dg} is inf-sup stable. Namely,
\begin{theorem}\label{thm:infsup} Let $(\bfu_h,p_h) \in \mcW_h \times \mcV_h$. {For sufficiently small $h$, there is a constant $C_s>0$} such that
\begin{equation}
\sup_{(\bfv_h, q_h) \in \mcW_h \times \mcV_h} \frac{A_h(\bfu_h, p_h; \bfv_h,q_h )+\varepsilon_u J_{\bfu}(\bfu_h,\bfv_h)+\varepsilon_p J_p(p_h,q_h)}{\tn (\bfv_h,q_h) \tn_h } 
\geq  C_s \tn (\bfu_h,p_h) \tn_h 
\end{equation}
{The constant $C_s$ is independent of $\mu_1$ and $\mu_2$ under assumption (\ref{eq:assumptionmu}).}
\end{theorem}
   
First, we show the coercivity of $a_h(\cdot,\cdot)$.  
\begin{lemma} \label{lem:coera}
There exists a constant $C_{coer}>0$ such that 
\begin{equation}
C_{coer} \tn \bfv_h \tn^2_h \leq a_h(\bfv_h, \bfv_h) + J_{\bfu}(\bfv_h,\bfv_h) \quad \forall \bfv_h \in \mcW_h.
\end{equation}
{The constant $C_{coer}$ is independent of $\mu_1$ and $\mu_2$ under assumption (\ref{eq:assumptionmu}).}
\end{lemma}
\begin{proof} 
Let $K_i=K\cap \Omega_i$. By the definition of $a_h(\cdot,\cdot)$ we have 
\begin{align}\label{eq:ah}
 a_h(\bfv_h, \bfv_h)&=\sum_{i=1}^2 \| \mu_i^{1/2} \bfeps( \bfv_{h,i} ) \|_{0,\Omega_i}^2 +
\sum_{K \in \hat{\mcK}_{\Gamma}}\left(\lambda_{\Gamma}|_{\Gamma\cap \overline{K}} \| \ldb \bfv_h \rdb \|^2_{0,\Gamma \cap \overline{K}}-2(\{ \mu \bfeps(\bfv_h) \bfn \},\ldb \bfv_h \rdb)_{\Gamma\cap \overline{K}} \right)\nonumber \\
&\quad +\sum_{i=1}^2\sum_{K \in \mcK_{\partial \Omega}}\left(\lambda_{\partial \Omega}|_{\partial \Omega \cap K_i} \| \bfv_{h,i} \|^2_{0,\partial \Omega \cap K_i}-2( \mu_i \bfeps(\bfv_{h,i})\bfn ,\bfv_{h,i})_{\partial \Omega \cap K_i} \right).
\end{align} 
{In the limit when $\mu_2\rightarrow 0^+$ the second sum on the right hand side, containing the interface terms, vanishes, see Remark~\ref{remmu}. Otherwise} 
for each $K\in {\mcK}_\Gamma$ let $| \Gamma \cap \overline{K}|=\gamma_K h_K$, $|K|=\alpha_Kh_k$, and $|K_i|=\alpha_{i,K}h_k$. Using the
Cauchy-Schwarz inequality and the geometric-arithmetic inequality we have 
\begin{align}\label{eq:secondterm}
&\sum_{K \in \hat{\mcK}_{\Gamma}}\left(\lambda_{\Gamma}|_{\Gamma\cap \overline{K}} \| \ldb \bfv_h \rdb \|^2_{0,\Gamma \cap \overline{K}}-2(\{ \mu \bfeps(\bfv_h)\bfn \},\ldb \bfv_h \rdb)_{\Gamma\cap \overline{K}} \right)\geq \nonumber \\
&\sum_{K \in \hat{\mcK}_{\Gamma}} \left(h_K\lambda_{\Gamma}|_{\Gamma\cap \overline{K}}-\frac{\delta_{1,K}+\delta_{2,K}}{\gamma_K}\right)h_K^{-1} \| \ldb \bfv_h \rdb \|^2_{0,\Gamma \cap \overline{K} }-\sum_{i=1}^2 \sum_{K \in \hat{\mcK}_{\Gamma}} \frac{|\Gamma\cap \overline{K}|}{\delta_{i,K}} \| \kappa_i \mu_i \bfeps(\bfv_{h,i})\bfn  \|^2_{0,\Gamma\cap \overline{K}}.
\end{align} 
{Using that $\mu_i=\bar{\mu}_i^{1/2} \{ \mu\}^{1/2}$ and the definition of $\|\cdot \|_{-1/2,h,\Gamma}$, the last term in the equation above satisfies
\begin{equation}\label{eq:negterm}
\sum_{i=1}^2 \sum_{K \in \hat{\mcK}_{\Gamma}} \frac{|\Gamma\cap \overline{K}|}{\delta_{i,K}} \| \kappa_i \mu_i \bfeps(\bfv_{h,i}) \bfn  \|^2_{0,\Gamma\cap \overline{K}} \geq \min_{i, K} \left( \frac{\{\mu \} \gamma_K}{\delta_{i,K}}\right) \| \{\bar{\mu}^{1/2} \bfeps(\bfv_h) \bfn \} \|_{-1/2,h,\Gamma}^2.
\end{equation}
Substituting equation~\eqref{eq:secondterm} into equation~\eqref{eq:ah}, for a constant $B>0$ adding and subtracting $B \lp \sum_{i=1}^2 \sum_{K \in \hat{\mcK}_{\Gamma}} \frac{|\Gamma\cap \overline{K}|}{\delta_{i,K}} \| \kappa_i \mu_i \bfeps(\bfv_{h,i}) \bfn  \|^2_{0,\Gamma\cap \overline{K}} \rp$, and using equation~\eqref{eq:negterm} we get 
\begin{align}
 a_h(\bfv_h, \bfv_h)&\geq \sum_{i=1}^2 \| \mu_i^{1/2} \bfeps( \bfv_{h,i} ) \|_{0,\Omega_i}^2 +
 \sum_{K \in \hat{\mcK}_{\Gamma}} \left(h_K\lambda_{\Gamma}|_{\Gamma\cap \overline{K}}-\frac{\delta_{1,K}+\delta_{2,K}}{\gamma_K}\right)h_K^{-1} \| \ldb \bfv_h \rdb \|^2_{0,\Gamma \cap \overline{K} }
  \nonumber \\
&\quad -(1+B) \sum_{i=1}^2 \sum_{K \in \hat{\mcK}_{\Gamma}} \frac{|\Gamma\cap \overline{K}|}{\delta_{i,K}} \| \kappa_i  \mu_i \bfeps(\bfv_{h,i}) \bfn  \|^2_{0,\Gamma\cap \overline{K}} 
\nonumber \\
&\quad+B\min_{i, K} \left( \frac{\{\mu \} \gamma_K}{\delta_{i,K}}\right) \| \{ \bar{\mu}^{1/2} \bfeps(\bfv_h) \bfn  \} \|_{-1/2,h,\Gamma}^2
 \nonumber \\
&\quad +\sum_{i=1}^2\sum_{K \in \mcK_{\partial \Omega}}\left(\lambda_{\Omega}|_{\partial \Omega \cap K_i} \| \bfv_{h,i} \|^2_{0,\partial \Omega \cap K_i}-2(\mu_i \bfeps(\bfv_{h,i}) \bfn,\bfv_{h,i})_{\partial \Omega \cap K_i} \right).
\end{align}  
} 

{For any constant $0<A<1$} we let  
\begin{equation}\label{eq:delta}
\delta_{1,K}= \frac{(1+B)\mu_1\kappa_1^2\gamma_{1}^2}{A \alpha_{1}}, \quad
\delta_{2,K}= \frac{(1+B)\mu_2\kappa_2^2\gamma_{2}^2}{A \alpha_{2}},
\end{equation}
with $\alpha_i=\alpha_{i,K}$ and $\gamma_{i}=\gamma_K$ for $i=1,2$, when $\Gamma\cap K$ intersects the boundary of the element exactly twice, and $\alpha_1=\alpha_{T}$, $\alpha_2=\alpha_{K}$, $\gamma_{1}=\gamma_T$, $\gamma_{2}=\gamma_K$ when $\Gamma \cap \overline{K}$ coincides with an edge shared by element $T\subset \Omega_1$ and $K\subset \Omega_2$. With our choices of $\kappa_1$, $\kappa_2$, and $\lambda_\Gamma$, using the inverse inequality in Lemma~\ref{lem:inverseineq}, $\mu_i=\bar{\mu}_i^{1/2} \{ \mu\}^{1/2}$, and that $\frac{ \{\mu \} \gamma_K}{\delta_{i,K}}\geq \tilde{C}>0$
 we obtain
\begin{align}\label{eq:ah1}
a_h(\bfv_h, \bfv_h) &\geq \sum_{i=1}^2 (1-A) \| \mu_i^{1/2} \bfeps( \bfv_{h,i}) \|^2_{0,\Omega_i} +
B  \tilde{C} \| \{ \bar{\mu}^{1/2} \bfeps(\bfv_h) \bfn  \} \|_{-1/2,h,\Gamma}^2 \nonumber \\
&\quad +\lp D+\lp C- \frac{(1+B)}{A}\rp \frac{1}{\gamma_K \lp \frac{\alpha_1}{\gamma_1^2}+ \frac{\alpha_2}{\gamma_2^2}\rp} \rp \|  \{\mu \}^{1/2}\ldb v_h \rdb \|^2_{1/2,h,\Gamma}
\nonumber \\
&\quad \quad +\sum_{i=1}^2\sum_{K \in \mcK_{\partial \Omega}}\left(\lambda_{\partial  \Omega}|_{\partial \Omega \cap K_i} \| \bfv_{h,i} \|^2_{0,\partial \Omega \cap K_i}-2( \mu_i \bfeps(\bfv_{h,i}) \bfn, \bfv_{h,i})_{\partial \Omega \cap K_i} \right).
\end{align}
Applying the Cauchy-Schwarz inequality, the geometric-arithmetic inequality, and the inverse inequality~\eqref{eq:invineqonB} on the boundary terms we get 
\begin{align}\label{eq:bterms}
\sum_{i=1}^2\sum_{K \in \mcK_{\partial \Omega}}\left(\lambda_{\partial \Omega}|_{\partial \Omega \cap \overline{K}_i} \| \bfv_{h,i} \|^2_{0,\partial \Omega \cap \overline{K}_i}-2( \mu \bfeps(\bfv_{h,i}) \bfn ,\bfv_{h,i})_{\partial \Omega \cap \overline{K}_i} \right) \geq \nonumber \\
\sum_{i=1}^2\min_{K\in \mcK_{\partial \Omega}} \left(h_K\lambda_{\partial \Omega}|_{\partial \Omega \cap \overline{K}_i}-\frac{\delta_{\partial \Omega,i,K}}{\gamma_{\partial \Omega,K}}\right) \| \bfv_{h,i}  \|^2_{1/2,h,\partial \Omega \cap \Omega_i}-
\nonumber \\
\sum_{i=1}^2 \sum_{K \in \mcK_{\partial \Omega}} \frac{\gamma_{\partial \Omega,K} h_K}{\delta_{\partial \Omega,i,K}} \| \mu_i \bfeps(\bfv_{h,i}) \bfn  \|^2_{0,\partial \Omega \cap \overline{K}_i},
\end{align} 
where $ \gamma_{{\partial \Omega},K} h_K=|\partial \Omega \cap \overline{K}|$. For each $K\in \mcK_{\partial \Omega}$ the following inverse inequality holds
\begin{equation}\label{eq:invineqonB}
h_K\| \bfeps(\bfv_{h,i}) \bfn   \|^2_{0,{\partial \Omega \cap \overline{K}_i}} \leq \frac{\gamma_{\partial \Omega, K}}{\alpha_K} \| \bfeps(\bfv_{h,i}) \|_{0,K}^2.
\end{equation}
For elements $K\in \mcK_{\partial \Omega}$ such that $K \cap \Omega_i \neq \emptyset$, $K \not \subset \Omega_i$, let $\mcF_{K,K^i}$ be the set of all faces that have to be crossed to pass from $K$ to $K^i\subset \Omega_i$ and $N_F$ the number of such faces. 
Our assumptions guarantee that such an element $K^i$ exists and that there are a bounded number of faces in $\mcF_{K,K^i}$. We use the same idea as in~\cite{BH12} and write that
\begin{equation}
\bfeps(\bfv_{h,i}) |_K= \bfeps(\bfv_{h,i})|_{K^i}+\sum_{F\in \mcF_{K,K^i}}\delta \ldb  \bfeps(\bfv_{h,i}) \bfn_F   \rdb_F \bfn_F,  
\end{equation} 
where $\delta=\pm 1$ with the sign depending on the orientation of $\bfn_F$ so that the equality holds. We then have
\begin{align}
\|\bfeps(\bfv_{h,i}) \|^2_{0,K}
\leq 2  \left(\frac{|K|}{|K^i|} \| \bfeps(\bfv_{h,i}) \|_{0,K^i}^2+ N_F \frac{|K|}{|F|} \sum_{F\in \mcF_{K,K^i}} \| \ldb  \bfeps(\bfv_{h,i}) \bfn_F \rdb_F \|_{0,F}^2 \right),
\end{align} 
where we have used the Cauchy-Schwarz inequality and the geometric-arithmetic inequality. 
Due to Assumption 1 (quasi-uniformity) we have that there exists a constant 
$c_q=\max{\left(\frac{|K|}{|K^i|},  \frac{|K|}{|F|h}\right)}$ and hence
\begin{equation}\label{eq:boundongrad}
\| \bfeps (\bfv_{h,i}) \|_{0,K}^2\leq 2c_q\left( \|\bfeps (\bfv_{h,i})\|_{0,K^i}^2+N_F h \sum_{F \in  \mcF_{K,K^i}} \| \ldb  \bfeps( \bfv_{h,i} ) \bfn_F  \rdb_F \|_{0,F}^2  \right).
\end{equation} 
Let $N_{K^i}$ be the number of elements $K \in \mcK_{\partial \Omega}$, $K \not \subset \Omega_i$ that have $K^i$ as the closest element completely in $\Omega_i$.  {Let $E$ and $F$ be positive constants, $E<\frac{1-A}{1+\max(N_{K^i})}$, and } 
\begin{equation}\label{eq:deltaB}
\delta_{\partial \Omega,i,K}= \frac{2c_q(1+F)\mu_i \gamma_{\partial \Omega, K}^2}{E \alpha_{K}} \quad i=1,2,
\end{equation}
when $K \cap \Gamma$ intersects the boundary of the element exactly twice and otherwise
 \begin{equation}\label{eq:deltaBinters}
\delta_{\partial \Omega,i,K}= \frac{(1+F)\mu_i \gamma_{\partial \Omega, K}^2}{E \alpha_{K}} \quad i=1,2.
\end{equation}
We get from our choice of $\lambda_{\partial \Omega}$ (equation~\eqref{eq:bterms}), the inverse inequality~\eqref{eq:invineqonB}, equation~\eqref{eq:boundongrad}, and
$\left(\frac{\mu_i \gamma_{\partial \Omega,K}}{\delta_{\partial \Omega, i,K}} \right)  \geq  \hat{C}>0$ that
\begin{align}
 a_h(\bfv_h, \bfv_h)&\geq \sum_{i=1}^2 (1-A-E(1+\max(N_{K^i}))) \| \mu_i^{1/2} \bfeps( \bfv_{h,i} ) \|^2_{0,\Omega_i} +
B \tilde{C} \| \{ \bar{\mu}^{1/2} \bfeps(\bfv_h) \bfn  \} \|_{-1/2,h,\Gamma}^2 \nonumber \\
&+\lp D+\lp C- \frac{(1+B)}{A}\rp \frac{1}{\gamma_K \lp \frac{\alpha_1}{\gamma_1^2}+ \frac{\alpha_2}{\gamma_2^2}\rp} \rp \|  \{\mu \}^{1/2}\ldb v_h \rdb \|^2_{1/2,h,\Gamma}
\nonumber \\
&+F  \hat{C} \|\mu^{1/2} \bfeps(\bfv_h) \bfn  \|^2_{-1/2,h,\partial \Omega}
+\lp G+(H-J)\frac{ \gamma_{\partial \Omega, K}}{\alpha_K} \rp \|  \mu^{1/2} \bfv_h \|^2_{1/2,h,\partial \Omega}
\nonumber \\
&- \sum_{i=1}^2 EN_F \sum_{\substack{
K \in \mcK_{\partial \Omega},\\
K \cap \Omega_i\neq \emptyset, \\
 K\not \subset \Omega_i
}}
 \sum_{F \in  \mcF_{K,K^i}} \mu_ih\| \ldb \bfeps(\bfv_{h,i}) \bfn_F  \rdb_F \|_{0,F}^2,
\end{align}
{where $J=\frac{2c_q(1+F)}{E}$ when the interface intersects the boundary of the element exactly twice and  otherwise $J=\frac{(1+F)}{E}$.  } 
Note that for each face in $F \in  \mcF_{K,K^i}$ we can write
\begin{align}\label{eq:epsandgrad}
\bfeps(\bfv_{h,i}) |_{K_F^+}= \bfeps(\bfv_{h,i})|_{K_F^-}+ \ldb  \bfeps(\bfv_{h,i}) \bfn_F  \rdb_F \bfn_F, \nonumber \\
\nabla(\bfv_{h,i})|_{K_F^+}= \nabla (\bfv_{h,i})|_{K_F^-}+ \ldb  \nabla(\bfv_{h,i}) \bfn_F  \rdb_F \bfn_F,
\end{align} 
and that each of the terms in equation~\eqref{eq:epsandgrad} are constant $2\times 2$ matrices. Hence, 
\begin{align}
 \sum_{F \in  \mcF_{K,K^i}} \| \ldb \bfeps(\bfv_{h,i}) \bfn_F  \rdb_F \|_{0,F}^2 \leq \sum_{F \in  \mcF_{K,K^i}} \| \ldb \nabla(\bfv_{h,i}) \bfn_F \rdb_F \|_{0,F}^2.
\end{align}
Finally, since $EN_F\leq \frac{(1-A)N_F}{1+\max(N_{K^i})}\leq1$ we have
\begin{equation}
 \sum_{i=1}^2 EN_F \sum_{\substack{
K \in \mcK_{\partial \Omega}, \\ K \cap \Omega_i\neq \emptyset, \\
 K\not \subset \Omega_i
}} \sum_{F \in  \mcF_{K,K^i}} \mu_ih\| \ldb  \bfeps(\bfv_{h,i}) \bfn_F  \rdb_F \|_{0,F}^2 \leq J_{\bfu}(\bfv_h,\bfv_h)
\end{equation}
{and coercivity follows if the constants $C$ and $H$ in $\lambda_{\Gamma}$ and $\lambda_{\partial \Omega}$, respectively are chosen such that $C\geq \frac{1+B}{A}$ and $H\geq J$.} 
\end{proof}

We will need the following technical lemma. 
\begin{lemma}\label{lem:technical} Let $p_h=(p_{h,1},p_{h,2}) \in\mcV_{h}$. There is a constant $C>0$ such that
\begin{equation}
\| \mu_i^{-1/2} p_{h,i} \|_{0,\omegahi}^2 \leq C\left(\| \mu_i^{-1/2} p_{h,i}\|_{0,\omegain }^2 + J_p(p_h,p_h) \right).
\end{equation}
\end{lemma}
\begin{proof}
For elements $K\in \mcK_{i}$  that are not entirely in $\Omega_i$, let $\mcF_{K,K^i}$ be the set of all faces that has to be crossed to pass from $K$ to the closest element $K^i\subset \Omega_i$ and $N_F$ the number of such faces. Assumption 4 guarantees that such an element $K^i$ exists and since the mesh is assumed to be shape regular there are a bounded number of faces in $\mcF_{K,K^i}$. We can write that
\begin{equation} \label{eq:pinomegai}
p_{h,i}|_K = p_{h,i}|_{K^i}+\sum_{F\in \mcF_{K,K^i}}\delta \ldb \bfn_F \cdot \nabla p_{h,i} \rdb _F \bfn_F \cdot (\bfx - \bfa_F), 
\end{equation} 
where $\delta=\pm 1$ with the sign depending on the orientation of $\bfn_F$ so that the equality holds and $\bfa_F$ is the center of gravity of $F$. Taking the square on both sides of identity~\eqref{eq:pinomegai}, using the Cauchy-Schwarz inequality and the geometric-arithmetic inequality, we get 
\begin{align}\label{eq:pKintermsofJ}
\|p_{h,i}\|_{0,K}^2 & \leq 2\left(\frac{|K|}{|K^i|}  \|p_{h,i}\|^2_{0,K^i} + N_F\sum_{F\in \mcF_{K,K^i}} \frac{|K|}{|F|} \| \ldb \bfn_F \cdot \nabla p_{h,i} \rdb_F \bfn_F \cdot (\bfx - \bfa_F)\|^2_{0,F}  \right)
\nonumber \\
&\leq 2c_q\left( \| p_{h,i} \|^2_{0,K^i} +  N_F \sum_{F\in \mcF_{K,K^i}}  h^3 \|  \ldb \bfn_F  \cdot \nabla p_{h,i} \rdb_F \|^2_{0,F} \right),
\end{align}
{where we have used that due to Assumption 1 (quasi-uniformity) $\frac{|K|}{|K^i|}$ is bounded by a constant and $\frac{|K|}{|F|} |\bfn_F \cdot (\bfx - \bfa_F)|^2 \leq c_qh^3$.}
Let $N_{K^i}$ be the number of elements in $\mcK_{\Gamma}$ that have $K^i$ as the closest element completely in $\Omega_i$. Summing over all elements $K\in \mcK_{i}$ and using equation~\eqref{eq:pKintermsofJ} for elements that are not entirely in $\Omega_i$ we obtain
\begin{align}
\|\mu_i^{-1/2}p_{h,i}\|_{0,\omegahi}^2 & \leq \sum_{K\in \mcK_i, K \subset \Omega_i} \left( 1+2c_qN_{K^i}  \right) \|\mu_i^{-1/2} p_{h,i} \|^2_{0,K^i} \nonumber \\
&+2c_q \max({N_F}) \sum_{K\in \mcK_i, K \not \subset \Omega_i} \sum_{F\in \mcF_{K,K^i}}  \mu_i^{-1}h^3 \|  \ldb \bfn_F  \cdot \nabla p_{h,i} \rdb_F \|^2_{0,F}.
\end{align} 
\end{proof}

To prove the inf-sup stability of $b_h(\cdot,\cdot)$ we use some of the ideas in~\cite{OR06}. Introduce the piecewise constant function
\begin{equation}
\overline{p}=\left\{ \begin{array}{ll}
\mu_1|\Omega_1|^{-1} & $\textrm{on} $ \Omega_1 \\
-\mu_2|\Omega_2|^{-1} & $\textrm{on} $ \Omega_2.
\end{array}\right.
\end{equation}
Let $M_0=$span$\{\overline{p}\}$. 
For any $p_h\in\mcV_h$ we can write
\begin{equation}
p_h=p_{0}+p_{h,0}^\perp, \quad p_{0}\in M_{0}, \ p_{h,0}^\perp \in M_{h,0}^\perp.
\end{equation}
The functions in $M_{h,0}^\perp$ satisfy $(p_{h,0}^\perp,1)_{\Omega_i}=0$, $i=1,2$, see \cite{OR06}.

\begin{lemma}\label{lem:infsupbp0} {For sufficiently small $h$, we have that for any $p_0 \in M_{0}$, there exists $\bfv_{h,0} \in \mcW_{h}$ and positive constants $ C_{1,p_0}$ and $C_{2,p_0}$ such that }
\begin{equation}
b_h(\bfv_{h,0}, p_0) \geq C_{1,p_0} \| \mu^{-1/2}p_0\|^2_{0,\Omega_1\cup \Omega_2}, \quad \tn \bfv_{h,0} \tn_h \leq C_{2,p_0}\| \mu^{-1/2} p_0\|_{0,\Omega_1\cup \Omega_2}.
\end{equation}
{The constants are independent of $\mu_1$ and $\mu_2$ under assumption (\ref{eq:assumptionmu}).}
\end{lemma}
\begin{proof}
Let $\tilde{p}_{0}=\mu^{-1} p_0$, then $(\tilde{p}_{0},1)_{\Omega_1 \cup \Omega_2}=0$. 
Let $I(\tilde{p_0})$ be the continuous piecewise linear approximation of $\tilde{p}_{0}$ which differs from $\tilde{p}_{0}$ only in elements $K\in K_\Gamma$.  
{Let  $q_h=I(\tilde{p}_{0})-\alpha$ where $\alpha=\frac{(I(\tilde{p}_{0}),1)_{\Omega_1\cup \Omega_2}}{\| 1\|_{0,\Omega_1\cup \Omega_2}^2}$ so that $(q_h,1)_{\Omega_1\cup \Omega_2}=0$.} 
Since the underlying finite element spaces are inf-sup stable there exist $\bfv_{h,0}\in \mcW_{h}\cap C(\Omega_1\cup \Omega_2)$ with $\bfv_{h,0}|_{\partial \Omega}=0$ such that 
\begin{equation}
\frac{b_h(\bfv_{h,0}, q_h)}{\|\nabla \bfv_{h,0} \|_{0,\Omega_1\cup \Omega_2} }\geq C\| q_h\|_{0,\Omega_1\cup \Omega_2}.
\end{equation}
We have 
\begin{align}\label{eq:infsup_ptilde}
\frac{b_h(\bfv_{h,0}, \tilde{p}_{0})}{\|\nabla \bfv_{h,0} \|_{0,\Omega_1\cup \Omega_2}}&=\frac{b_h(\bfv_{h,0},q_h)}{\|\nabla \bfv_{h,0} \|_{0,\Omega_1\cup \Omega_2}} + \frac{b_h(\bfv_{h,0}, \tilde{p}_{0}-q_h)}{\|\nabla \bfv_{h,0} \|_{0,\Omega_1\cup \Omega_2}}
\nonumber \\
&\geq  C \| q_h \|_{0,\Omega_1\cup \Omega_2}
- \sqrt{2} \| \tilde{p}_{0}-q_h \|_{0,\Omega_1\cup \Omega_2}
\nonumber \\
&\geq  C \| \tilde{p}_{0} \|_{0,\Omega_1\cup \Omega_2}-
( \sqrt{2}+C)\| \tilde{p}_{0}-q_h \|_{0,\Omega_1\cup \Omega_2} 
\nonumber \\
&\geq
\left(C-(C+ \sqrt{2}) \frac{\| \tilde{p}_{0}-q_{h} \|_{0,\Omega_1\cup \Omega_2}}{\| \tilde{p}_{0} \|_{0,\Omega_1\cup \Omega_2}} \right)
\| \tilde{p}_{0} \|_{0,\Omega_1\cup \Omega_2}
\nonumber \\
&\geq  \left(C-ch^ {1/2} \right)   \| \tilde{p}_{0} \|_{0,\Omega_1\cup \Omega_2},
\end{align}
where we in the last step have used that 
\begin{equation}
|\alpha|=\frac{\left| (I(\tilde{p}_{0})-\tilde{p}_{0},1)_{\Omega_1\cup \Omega_2} \right|}{\| 1 \|^2_{0,\Omega_1\cup \Omega_2}} \leq \frac{\| I(\tilde{p}_{0})-\tilde{p}_{0} \|_{0,\Omega_1\cup \Omega_2}}{\| 1 \|_{0,\Omega_1\cup \Omega_2}}
\end{equation}
and hence
\begin{equation}
 \frac{\| \tilde{p}_{0}-q_{h} \|_{0,\Omega_1\cup \Omega_2}}{\| \tilde{p}_{0} \|_{0,\Omega_1\cup \Omega_2}}\leq 2
 \frac{\| I(\tilde{p}_{0})-\tilde{p}_{0} \|_{0,\Omega_1\cup \Omega_2}}{\| \tilde{p}_{0} \|_{0,\Omega_1\cup \Omega_2}}\leq Ch^{1/2}.
\end{equation}
From the definition of $M_0$ one can see that
\begin{equation}
\| \mu^{-1/2}p_0\|^2_{0,\Omega_1\cup \Omega_2}=C(\mu,\Omega)\| \tilde{p}_{0}\|^2_{0,\Omega_1\cup \Omega_2}, \quad b_h(\bfv_{h,0}, p_0)=C(\mu,\Omega) b_h(\bfv_{h,0}, \tilde{p}_{0}),
\end{equation}
with $C(\mu,\Omega)=\frac{\mu_1|\Omega_1|^{-1}+\mu_2|\Omega_2|^{-1}}{|\Omega_1|^{-1}+|\Omega_2|^{-1}}$.
We can choose $\bfv_{h,0}$ so that equation~\eqref{eq:infsup_ptilde} is satisfied and $ \|\nabla \bfv_{h,0} \|_{0,\Omega_1\cup \Omega_2}= \| \tilde{p}_{0} \|_{0,\Omega_1\cup \Omega_2}$
and obtain 
\begin{equation}
b_h(\bfv_{h,0}, p_0)= C(\mu,\Omega) b_h(\bfv_{h,0}, \tilde{p}_{0}) \geq c C(\mu,\Omega)\| \tilde{p}_{0}\|^2_{0,\Omega_1\cup \Omega_2}=c \| \mu^{-1/2}p_0\|^2_{0,\Omega_1\cup \Omega_2}.
\end{equation}
We have
\begin{equation}
C(\mu,\Omega)\geq \min_{i=1,2} \left(  \frac{|\Omega_i |}{|\Omega_1|+|\Omega_2|}\right) \mu_{\max}=\tilde{C} \mu_{\max}, 
\end{equation}
and 
\begin{align}
\tn \bfv_{h,0} \tn &\leq 
C  \mu_{\max}^{1/2} \| \nabla \bfv_{h,0} \|_{0,\Omega_1\cup \Omega_2} =C\mu_{\max}^{1/2} \| \tilde{p}_{0}\|_{0,\Omega_1\cup \Omega_2}
\leq C\tilde{C}^{-1/2}   \| \mu^{-1/2} p_0\|_{0,\Omega_1\cup \Omega_2}. 
\end{align}
Finally, note that for $\bfv_{h,0}=(\bfv_{h,1}, \bfv_{h,2}) \in \mcW_{h}\cap C(\Omega_1\cup \Omega_2)$ we can choose $\bfv_{h,i}=\bfv_{h,j}$ in $\omegahi\cap \Omega_j$,  $j\neq i$ so that $\tn \bfv_{h,0} \tn_h\leq C\tn \bfv_{h,0} \tn $. 
\end{proof}
 
\begin{lemma}\label{lem:infsupbpi} {For sufficiently small $h$, we have that for any $p_{h,0}^\perp=(p_{h,1}^\perp,p_{h,2}^\perp) \in M_{h,0}^\perp$ there exists $\bar{\bfv}_h\in \mcW_h$ and positive constants $C_{1,p_{h,0}^\perp}$, $C_{2,p_{h,0}^\perp}$, and $C_{3,p_{h}^\perp}$ such that }
\begin{equation}
b_h(\bar{\bfv}_{h}, p_{h,0}^\perp) \geq C_{1,p_{h,0}^\perp} \|\mu^{-1/2} p_{h}^\perp \|_{0,\omegahO\cup \omegahT}^2 - C_{2,p_{h,0}^\perp} J_p(p_{h,0}^\perp,p_{h,0}^\perp)
\end{equation}
and 
\begin{equation}
\tn \bar{\bfv}_h \tn_h \leq C_{3,p_{h}^\perp} \|\mu^{-1/2} p_{h}^\perp \|_{0,\omegahO\cup \omegahT}.
\end{equation}
{The constants are independent of $\mu_1$ and $\mu_2$ under assumption (\ref{eq:assumptionmu}).}
\end{lemma}
\begin{proof}
{Let  $q_{h,i}=p_{h,i}^\perp-\alpha_i$ where $\alpha_i=\frac{(p_{h,i}^\perp,1)_{\omegain}}{\| 1\|_{0,\omegain}^2}$ so that $(q_{h,i},1)_{\omegain}=0$.} 
{Since the underlying finite element spaces are inf-sup stable with a uniform constant on any polygon of shape regular elements which has no element with two edges on the boundary, see Brezzi-Fortin \cite{BrFo91}, Proposition 6.1, Page 252, there is for each $q_h=(q_{h,1},q_{h,2})$  a $\bfv_{h,\Omega_i}=(\bfv_{h,1}, \bfv_{h,2}) \in \mcW_h$ with $\bfv_{h,i} \in \mcW_{h,i}$, $\bfv_{h,i}=0$ on $\omegahi \setminus \omegain$ and on $\partial \Omega$, and $\bfv_{h,j}=0$ for $j\neq i$ such that }
\begin{equation}\label{eq:infsup1}
\frac{b_h(\bfv_{h,\Omega_i}, q_h)}{\tn \bfv_{h,\Omega_i} \tn } \geq C \| \mu_i^ {-1/2} q_{h,i}\|_{0,\omegain}.
\end{equation}
Using the inverse inequality~\eqref{eq:traceineqF} and that $\supp(\bfv_{h,\Omega_i})=\omegain$
\begin{equation}
J( \bfv_{h,\Omega_i}, \bfv_{h,\Omega_i})\leq c_q^{-1}  \|  \mu^{1/2}\nabla \bfv_{h,i} \|_{0,\omegain}^2,
\end{equation} 
which together with Korn's inequality~\cite[Eq. (1.19)]{B03} yields
\begin{equation}\label{eq:normrel}
\tn \bfv_{h,\Omega_i} \tn_h^2= \tn \bfv_{h,\Omega_i} \tn^2 +J( \bfv_{h,\Omega_i}, \bfv_{h,\Omega_i})\leq C\tn \bfv_{h,\Omega_i} \tn^2.
\end{equation} 
We can choose  $\bfv_{h,\Omega_i}$ so that equation~\eqref{eq:infsup1} and~\eqref{eq:normrel} are satisfied and $\tn \bfv_{h,\Omega_i} \tn_h=\| \mu_i^ {-1/2} q_{h,i} \|_{0,\omegahi}$. We then have
\begin{equation}\label{eq:infsupinomegai}
b_h(\bfv_{h,\Omega_i}, q_h) \geq C_q \| \mu_i^ {-1/2} q_{h,i} \|_{0,\omegahi} \| \mu_i^ {-1/2} q_{h,i}\|_{0,\omegain}.
\end{equation}
Lemma~\ref{lem:technical} then yields
\begin{equation}\label{eq:infsupq}
\| \mu_i^{-1/2} q_{h,i}  \|_{0,\omegahi}^2 \leq C\left(\| \mu_i^{-1/2} q_{h,i}\|_{0,\omegain }^2 + J_p(q_h,q_h) \right) \leq 
C\left(C_q^ {-1}b_h(\bfv_{h,\Omega_i}, q_h) + J_p(q_h,q_h) \right).
\end{equation}
Note that $b_h(\bfv_{h,\Omega_i}, q_h)=b_h(\bfv_{h,\Omega_i}, p_{h,0}^\perp)$, $J_p(q_h,q_h)=J_p(p_{h,0}^\perp,p_{h,0}^\perp)$ and
\begin{equation}\label{eq:absalpha}
|\mu_i^{-1/2}\alpha_i|=\frac{|(\mu_i^{-1/2}\alpha_i,1)_{\omegain}|}{\|1\|_{0,\omegain}^2}= \frac{|(\mu_i^{-1/2} p_{h,i}^\perp,1)_{\Omega_i \setminus \omegain}|}{\|1\|_{0,\omegain}^2} \leq
 \frac{\|\mu_i^{-1/2} p_{h,i}^\perp\|_{0,\omegahi} \|1\|_{0, \Omega_i \setminus \omegain}}{\|1\|_{0,\omegain}^2}. 
\end{equation}
Using equation~\eqref{eq:absalpha} we get
\begin{align}\label{eq:normqh}
\| \mu_i^{-1/2} q_{h,i}  \|_{0,\omegahi}^2 & \geq  \| \mu_i^{-1/2} p_{h,i}^\perp  \|_{0,\omegahi}^2- \| \mu_i^{-1/2} \alpha_i  \|_{0,\omegahi}^2
\nonumber \\
 &\geq\| \mu_i^{-1/2} p_{h,i}^\perp  \|_{0,\omegahi}^2\left(1- \frac{\|1\|_{0,\omegahi}^2 \|1\|_{0, \Omega_i \setminus \omegain}}{\|1\|_{0,\omegain}^2} \right).
\end{align}
We assume $\|1\|_{0, \Omega_i \setminus \omegain}=ch^{1/2}$ and hence equation~\eqref{eq:infsupq} and~\eqref{eq:normqh} yield
\begin{equation}
\| \mu_i^{-1/2} p_{h,i}^\perp  \|_{0,\omegahi}^2 \leq 
C(1-ch^ {1/2})^{-1}\left(C_q^ {-1} b_h(\bfv_{h,\Omega_i}, p_{h,0}^\perp)+J_p(p_{h,0}^\perp,p_{h,0}^\perp) \right).
\end{equation}
From equation~\eqref{eq:absalpha} we also obtain $\tn \bfv_{h,\Omega_i} \tn_h\leq C \| \mu_i^ {-1/2} p_{h,i}^\perp \|_{0,\omegahi}$.
Finally, taking $\bar{\bfv}_h= \bfv_{h,\Omega_1}+ \bfv_{h,\Omega_2}$ we have
\begin{equation}
 b_h(\bar{\bfv}_h,p_{h,0}^\perp) \geq  C_1 \sum_{i=1}^2 \| \mu_i^ {-1/2} p_{h,i}^\perp \|_{0,\omegahi}^2-C_2J_p(p_{h,0}^\perp,p_{h,0}^\perp)
\end{equation}
and  $\tn \bar{\bfv}_h \tn_h\leq C \sum_{i=1}^2\| \mu_i^ {-1/2} p_{h,i}^\perp \|_{0,\omegahi}$.
\end{proof}

\begin{lemma}\label{lem:infsupb} {For sufficiently small $h$, we have that for any $p_h\in \mcV_h$ there exists $\bfv_h \in \mcW_h$ and constants $C_1,C_3>0$, and $C_2 \geq 0$ such that  }
\begin{equation}
b_h(\bfv_h, p_h) \geq C_1  \|\mu^{-1/2} p_{h} \|_{0,\omegahO\cup \omegahT}^2 - C_2 J_p(p_h,p_h), \quad \tn \bfv_h \tn_h \leq C_3 \|\mu^{-1/2} p_h \|_{0,\omegahO\cup \omegahT }. 
\end{equation}
{The constants are independent of $\mu_1$ and $\mu_2$ under assumption (\ref{eq:assumptionmu}).}
\end{lemma}
\begin{proof}
If $p_h$ is piecewise constant, {i.e. $p_h\in M_0$}, the lemma follows from Lemma~\ref{lem:infsupbp0} with $C_2=0$. 
Otherwise, we have $p_h=p_0+p_{h,0}^\perp$, where $p_0 \in M_0$ and $p_{h,0}^\perp \in M_{0,h}^\perp$. Let $\bfv_{h,0}$ be such that Lemma~\ref{lem:infsupbp0} is satisfied 
and $\bar{\bfv}_h$ such that Lemma~\ref{lem:infsupbpi} is satisfied. 
For $\alpha>0$, define $\bfv_h=\bfv_{h,0}+\alpha \bar{\bfv}_h$.
Note that $b_h( \bar{\bfv}_h, p_0)=0$, since $\bar{\bfv}_h$ vanishes on $\Gamma \cup \partial \Omega$ and  $p_0$ is constant on each subdomain $\Omega_i$, $i=1,2$, and
\begin{equation}
| b_h(\bfv_{h,0},p_{h,0}^\perp)|=\left | (\nabla \cdot \bfv_{h,0}, p_{h,0}^\perp)_{\Omega_1\cup \Omega_2} \right| \leq C \tn \bfv_{h,0} \tn  \|\mu^{-1/2} p_{h,0}^\perp \|_{0,\Omega_1\cup \Omega_2},
\end{equation}
since $\bfv_{h,0}$ is continuous and vanishes on $\partial \Omega$. Also, $J_p(p_h,p_h)=J_p(p_{h,0}^\perp ,p_{h,0}^\perp )$.
Thus,  
\begin{align}
b_h(\bfv_h,p_h)&=
b_h(\bfv_{h,0}, p_0)+
b_h(\bfv_{h,0},p_{h,0}^\perp)
+\alpha b_h( \bar{\bfv}_h, p_0)
+\alpha b_h( \bar{\bfv}_h, p_{h,0}^\perp)
\nonumber \\
& \geq C_{1,p_0}\| \mu^{-1/2}p_0 \|_{0,\Omega_1\cup \Omega_2}^2-CC_{2,p_0}\| \mu^{-1/2}p_0 \|_{0,\Omega_1\cup \Omega_2}\|  \mu^{-1/2}p_{h,0}^\perp \|_{0,\omegahO\cup \omegahT}
\nonumber \\
\quad &+\alpha \left(C_{1,p_{h,0}^\perp}  \|\mu^{-1/2} p_{h,0}^\perp  \|_{0,\omegahO\cup \omegahT} ^2- C_{2,p_{h,0}^\perp} J_p(p_{h,0}^\perp ,p_{h,0}^\perp ) \right)
\nonumber \\
& {\geq \ \left(C_{1,p_0}-CC_{2,p_0}\beta/2 \right) \| \mu^{-1/2}p_0 \|_{0,\Omega_1\cup \Omega_2}^2
}
\nonumber \\
&{+\left( \alpha C_{1,p_{h,0}^\perp}-CC_{2,p_0}/(2\beta)  \right) \| \mu^{-1/2} p_{h,0}^\perp  \|_{0,\omegahO\cup \omegahT}^2- \alpha C_{2,p_{h,0}^\perp} J_p(p_{h,0}^\perp ,p_{h,0}^\perp)}
\nonumber \\
&
\geq C_1 \| \mu^{-1/2} p_{h} \|_{0,\omegahO\cup \omegahT}^2-C_2 J_p(p_{h} ,p_{h} )
\end{align} 
for sufficiently large $\alpha$. Finally, we also have
\begin{align}
\tn \bfv_h \tn_h \leq \tn \bfv_{h,0} \tn_h +\alpha \tn  \bar{\bfv}_h \tn_h 
&\leq C \| \mu^{-1/2} p_{h} \|_{0,\omegahO\cup \omegahT}.
\end{align}
\end{proof}

We are now ready to prove the inf-sup theorem.
\begin{proof} (of Theorem \ref{thm:infsup})
Note that if $p_h\in \mcV_h$ is constant, $p_h=0$ since
$(\mu^{-1}p_h,1)_{\Omega_1 \cup \Omega_2}=0$. Letting $\bfv_h=\bfu_h$
and using the coercivity of $a_h(\cdot,\cdot)$ we have
\begin{align}
&A_h(\bfu_h, p_h ; \bfu_h, p_h)+\varepsilon_{\bfu}J_{\bfu}(\bfu_h,\bfu_h)+\varepsilon_{p}J_p(p_h,p_h)=a_h(\bfu_h, \bfu_h)+\varepsilon_{\bfu}J_{\bfu}(\bfu_h,\bfu_h)  \nonumber \\
&\geq C_{coer}\min(1,\varepsilon_{\bfu}) \tn \bfu_h \tn^2_h=C_{coer}\min(1,\varepsilon_{\bfu}) \tn (\bfu_h,p_h) \tn^2_h,
\end{align} 
and hence the proof follows. Otherwise, let $\bfv_h$ be such that Lemma \ref{lem:infsupb} is satisfied and $\tn \bfv_h \tn_h = \| \mu^{-1/2}p_h \|_{0,\omegahO \cup \omegahT}$. Then, using the coercivity and continuity of $a_h( \cdot, \cdot)$ (Lemma~\ref{lem:conta}, and~\ref{lem:coera}), Cauchy--Schwarz inequality, and the stability of $b_h( \cdot, \cdot)$ (Lemma~\ref{lem:infsupb}) we have for $\alpha>0$
\begin{align}\label{eq:infsupproof}
&A_h(\bfu_h, p_h ; \bfu_h - \alpha \bfv_h,p_h)+\varepsilon_{\bfu}J_{\bfu}(\bfu_h,\bfu_h- \alpha \bfv_h)+\varepsilon_{p}J_p(p_h,p_h)=
 \nonumber \\
&a_h(\bfu_h,\bfu_h)+\varepsilon_{\bfu}J_{\bfu}(\bfu_h,\bfu_h)-\alpha \left(a_h(\bfu_h,\bfv_h)+ \varepsilon_{\bfu}J_{\bfu}(\bfu_h,\bfv_h) \right)
+\alpha b_h(\bfv_h,p_h)+\varepsilon_{p}J_p(p_h,p_h) \geq
\nonumber \\
&C_{coer}\min(1,\varepsilon_{\bfu}) \tn \bfu_h \tn^2_h -\alpha \max(C_{cont},\varepsilon_{\bfu}) (\tn \bfu_h \tn + (J_{\bfu}(\bfu_h,\bfu_h))^{1/2}) 
(\tn \bfv_h \tn + (J_{\bfu}(\bfv_h,\bfv_h))^{1/2}) 
\nonumber \\
&+\alpha
\left( C_1  \| \mu^{-1/2}p_h \|_{0,\omegahO \cup \omegahT}^2 - C_2 J_p(p_h,p_h) \right)
+\varepsilon_{p}J_p(p_h,p_h) \geq 
D  \tn (\bfu_h ,p_h ) \tn_h^2.
\end{align} 
with $D=\min({D_1,D_2,D_3})$, where
\begin{align}
D_1 &= \left(C_{coer}\min(1,\varepsilon_{\bfu})  -\alpha \max(C_{cont},\varepsilon_{\bfu})/\delta \right) >0, \nonumber \\
D_2 &= \alpha \left( C_1 - \max(C_{cont},\varepsilon_{\bfu}) \delta \right)>0, \nonumber \\
D_3 &= \left( \varepsilon_p- \alpha C_2 \right)> 0,
\end{align}
provided $\delta$ and $\alpha$ are sufficiently small. Finally, the proof follows using that 
\begin{equation}
\tn (\bfu_h -\alpha \bfv_h,p_h ) \tn_h \leq  \tn (\bfu_h ,p_h ) \tn_h + \alpha \tn \bfv_h \tn_h 
\leq  (1 +\alpha)  \tn (\bfu_h ,p_h ) \tn_h
\end{equation}
in equation~\eqref{eq:infsupproof}.
\end{proof}

\subsection{Approximation properties}
In this Section we will show that the spaces $\mcV_h$ and $\mcW_h$ have optimal approximation properties on $H^1(\Omega_1 \cup \Omega_2)$ and $[H^2(\Omega_1 \cup \Omega_2)]^2$, respectively, in the energy norm.  
In order to construct an interpolation operator we recall that there is an extension operator $\mcE_i^s:H^s(\Omega_i) \rightarrow H^s(\Omega)$, $i=1,2$, $s\geq 0$, 
such that $\mcE_i^s w_i|_{\Omega_i}=w_i$ and 
\begin{equation}\label{eq:stabext}
\| \mcE_i^s w_i \|_{s,\Omega} \leq C \| w_i \|_{s,\Omega_i} \quad \forall w_i \in H^s(\Omega_i).
\end{equation}
See~\cite{Dautray} for further details. Let $\pi_{h}: H^s(\Omega) \rightarrow V_{h,0}$, {where $V_{h,0}=\mcW_{h,0}$ for the velocity and 
$V_{h,0}=\mcV_{h,\mu}$ for the pressure}, be the standard Scott-Zhang 
interpolation operator~\cite{EG04} and recall the stability property 
\begin{equation}\label{eq:stabintop}
\| \pi_{h} w\|_{r,\Omega} \leq C \| w \|_{s,\Omega}, \quad  0\leq r \leq \min{(1,s)}, \quad  \forall w \in H^s(\Omega)
\end{equation}
and the approximation property of the interpolation operator
\begin{equation}\label{eq:approxpI}
\| w - \pi_{h} w\|_{r,K} \leq C h_K^{s-r}| w |_{s,\mcN(K)}, \quad 0\leq r \leq s \leq 2, \quad \forall K\in \mcK,  \ \forall w \in H^s(\Omega),
\end{equation}
where $\mcN(K)$ is the set of elements in $\mcK$ sharing at least one vertex 
with $K$. We define 
\begin{equation}
\pi_{h,i}^* w_i=\pi_{h} \mcE_i^s w_i |_{\omegahi} \quad \forall w_i\in H^s(\Omega_i)
\end{equation}
and for $w=(w_1,w_2)$ with $w_i|_{\Omega_i}\in H^s(\Omega_i) $ we define 
\begin{equation}
\pi_{h}^* w= (\pi_{h,1}^*w_1, \pi_{h,2} ^*w_2 ).
\end{equation}
We will use the same interpolant for the velocity and pressure. For the velocity $s=2$, $V_{h,0}=\mcW_{h,0}$, and $\pi_{h,i}^*: H^2(\Omega_i) \rightarrow \mcW_{h,i}$, $i=1,2$ while for the pressure  $s=1$, $V_{h,0}=\mcV_{h,\mu}$  and $\pi_{h,i}^*: H^1(\Omega_i) \rightarrow \mcV_{h,i}$, $i=1,2$.
In the norm $\tn (\cdot,\cdot) \tn$, we have the following interpolation error estimate:
\begin{lemma} \label{lem:interpolest}
It holds that
\begin{equation} 
\tn (\bfv - \pi_h^* \bfv, p - \pi_h^* p) \tn^2 \leq  h^2 (C_u \|\mu_{\max}^{1/2} \bfv \|^2_{2,\Omega_1\cup \Omega_2} +C_p\|\mu^{-1/2} p \|_{1,\Omega_1\cup \Omega_2}^2 ),
\end{equation} 
{where $C_u$ and $C_p$ are positive constants independent of $\mu_1$ and $\mu_2$ under assumption (\ref{eq:assumptionmu}).}
\end{lemma}
\begin{proof}
Recall the definition of the norm $\tn (\cdot,\cdot) \tn$ (equation~\eqref{eq:tnormup}). The interface and boundary contributions can be estimated in terms of element contributions by applying the trace inequalities in Lemma~\ref{lem:traceineq}. Then, for the element contributions, applying the approximation property of the interpolation operator~\eqref{eq:approxpI}, and finally using the stability of the extension operator, equation~\eqref{eq:stabext}, yields the desired estimate. {We also use that $\kappa_i\bar{\mu_i}^{1/2}=\frac{\kappa_i\mu_i}{\sqrt{\kappa_1\mu_1+\kappa_2\mu_2}} \leq \mu_i^{1/2}$ in the estimate of $\| \{ \bar{\mu}^{1/2}  \bfeps(\bfv - \pi_h^* \bfv)  \bfn   \} \|^2_{-1/2,h,\Gamma}$, 
that $\{ \mu\}^{1/2} \leq \mu_{\max}^{1/2}$ in the estimate of  $\|\{ \mu\}^{1/2} [\bfv - \pi_h^* \bfv]\|^2_{1/2,h,\Gamma}$, and that $\kappa_i\{ \mu\}^{-1/2}=\frac{\kappa_i\mu_i^{1/2} \mu_i^{-1/2}}{\sqrt{\kappa_1\mu_1+\kappa_2\mu_2}}\leq \mu_i^{-1/2}$ in the estimate of $\| \{ \mu\}^{-1/2} \{ p - \pi_h^* p \} \|^2_{-1/2,h,\Gamma} $. }
\end{proof}

\subsection{A priori error estimates}
We have the following error estimate:
\begin{theorem} It holds that
\begin{equation}
\tn (\bfu - \bfu_h, p - p_h) \tn  \leq C h \left(  \| \mu_{\max}^{1/2} \bfu \|_{2,\Omega_1\cup \Omega_2} +\| \mu^{-1/2}p\|_{1,\Omega_1\cup \Omega_2} \right),
\end{equation}
 {where C is a positive constant independent of $\mu_1$ and $\mu_2$ under assumption (\ref{eq:assumptionmu}).}
\end{theorem}
\begin{proof}
We have 
\begin{equation}\label{eq:usetriangineq}
\tn (\bfu - \bfu_h, p - p_h) \tn \leq \tn (\bfu - \pi_h^* \bfu, p - \pi_h^* p) \tn +  \tn (\pi_h^* \bfu-\bfu_h ,  \pi_h^* p -p_h ) \tn_h.
\end{equation} 
Here the first term can be estimated directly using the interpolation error 
estimate (Lemma \ref{lem:interpolest}) 
\begin{equation}\label{eq:useintperror}
\tn (\bfu - \pi_h^* \bfu, p - \pi_h^* p) \tn \leq C h \left( \| \mu_{\max}^{1/2}\bfu \|_{2,\Omega_1\cup \Omega_2} +  \|\mu^{-1/2}p\|_{1,\Omega_1\cup \Omega_2} \right).
\end{equation}
Turning to the second term we use the inf-sup condition (Theorem~\ref{thm:infsup}) followed 
by the consistency relation, Lemma~\ref{l:Galorth}, to get
\begin{align}
&\tn (\pi_h^* \bfu-\bfu_h ,  \pi_h^* p -p_h ) \tn_h  \nonumber \\
& \quad \leq 
C_s^{-1}\sup_{(\bfv, q) \in \mcW_h \times \mcV_h} \frac{A_h( \pi_h^* \bfu-\bfu_h, \pi_h^* p -p_h ; \bfv_h,q_h)+\varepsilon_{\bfu}J(\pi_h^* \bfu-\bfu_h, \bfv_h)+\varepsilon_{p}J(\pi_h^* p -p_h ,q_h)}{\tn (\bfv_h,q_h) \tn_h }
 \nonumber \\
& \quad \leq C_s^{-1} \sup_{(\bfv, q) \in \mcW_h \times \mcV_h} \frac{A_h(\pi_h^* \bfu - \bfu , \pi_h^* p - p; \bfv_h,q_h )+\varepsilon_{\bfu}J(\pi_h^* \bfu, \bfv_h)+\varepsilon_{p}J(\pi_h^* p,q_h) }{\tn (\bfv_h,q_h) \tn_h } 
\end{align}
{By the Cauchy-Schwarz inequality we have that}
\begin{equation}
J(\pi_h^* \bfu, \bfv_h)\leq J(\pi_h^* \bfu, \pi_h^* \bfu)^{1/2} J_{\bfu}(\bfv_h,\bfv_h)^{1/2}\leq  J(\pi_h^* \bfu, \pi_h^* \bfu)^{1/2} \tn (\bfv_h,q_h) \tn_h 
\end{equation}
{and similarly for $J(\pi_h^* p,q_h)$. By continuity of $A_h(\cdot,\cdot; \cdot, \cdot)$, Lemma \ref{lem:contb}, it therefore follows that}
\begin{align}\label{eq:error}
\tn (\pi_h^* \bfu&-\bfu_h ,  \pi_h^* p -p_h ) \tn_h  \leq  \nonumber \\
&C_s^{-1}\left( C_{A}\tn (\bfu - \pi_h^* \bfu, p - \pi_h^* p) \tn + \varepsilon_{\bfu}J(\pi_h^* \bfu, \pi_h^* \bfu)^{1/2}+\varepsilon_{p}J(\pi_h^* p, \pi_h^* p)^{1/2}\right).
\end{align}
The first term is estimated using the interpolation error estimate. 
For $\bfu_i \in H^2(\Omega_i)$ and $\mcE^2 \bfu=(\mcE^2_1 \bfu_1,\mcE^2_2 \bfu_2)$ we have that 
\begin{align}\label{eq:termstabu}
J(\pi_h^* \bfu, \pi_h^* \bfu)&=
J(\mcE^2 \bfu-\pi_h^* \bfu, \mcE^2 \bfu-\pi_h^* \bfu) 
\nonumber \\
&\leq \sum_{i=1}^2 \sum_{F \in \mcF_{\Gamma,i}} C\mu_i h
 \| \ldb  \nabla (\mcE^2_i \bfu_i-\pi_{h,i}^* \bfu_i) \bfn_F  \rdb_F  \|_{0,F}^2 
\nonumber \\
&\leq \sum_{i=1}^2 C\mu_i\sum_{K \in \mcK} \left( \|\mcE^2_i \bfu_i-\pi_{h,i}^* \bfu_i\|_{1,K}^2
+ h^{2}\|\mcE^2_i \bfu_i-\pi_{h,i}^* \bfu_i\|_{2,K}^2 \right)
\nonumber \\
& \leq C h^2 \sum_{i=1}^2 \mu_i \sum_{k \in \mcK} \|\mcE^2_i \bfu_i\|_{2,K}^2
\leq  C h^2 \sum_{i=1}^2 \mu_i\| \bfu_i \|_{2,\Omega_i}^2,
\end{align}
where we have used the Cauchy-Schwarz inequality, the trace inequality in Lemma~\ref{lem:traceineq}, the approximation property of the interpolation operator (equation~\eqref{eq:approxpI}), and finally the stability of the extension operator (equation~\eqref{eq:stabext}).

Third term in equation~\eqref{eq:error} can be estimated using the inverse estimate
\begin{equation}
h\|  \ldb \bfn \cdot \nabla \pi_{h,i}^* p_i \rdb \|_{0,F}^2 \leq C \| \nabla  \pi_{h,i}^* p_i \|_{0,K_F^+ \cup  K_F^-}^2, 
\end{equation}
where $K_F^+$ and $K_F^-$ are the elements sharing face F. The inverse estimate together with the stability property of $\pi_h$ equation~\eqref{eq:stabintop} and the stability of the extension operator (equation~\eqref{eq:stabext}) yield 
\begin{align}\label{eq:termstabp}
J(\pi_h^* p, \pi_h^* p) &= \sum_{i=1}^2 \sum_{F \in \mcF_{\Gamma,i}} C\mu_i^{-1} h^3  \| \ldb  \bfn_F \cdot \nabla \pi_{h,i}^* p_i  \rdb  \|_{0,F}^2 
\nonumber \\
&\leq C h^2 \sum_{i=1}^2\mu_i^{-1}  \sum_{F \in \mcF_{\Gamma,i}}  \| \nabla \pi_{h} \mcE_i^1p_i \|_{0,K_F^+\cup  K_F^-}^2 \leq C h^2 \sum_{i=1}^2 \mu_i^{-1} \sum_{k \in \mcK}  \|\pi_{h}\mcE_i^1p_i \|_{1,K}^2
\nonumber \\
&\leq  C h^2 \sum_{i=1}^2 \mu_i^{-1}  \| p_i \|_{1,\Omega_i}^2. 
\end{align}

Collecting the estimates~\eqref{eq:usetriangineq},~\eqref{eq:useintperror}~\eqref{eq:error},~\eqref{eq:termstabu}, and~\eqref{eq:termstabp} the theorem follows.
\end{proof}
An $L^2$-estimate for the velocity can be proven assuming additional regularity and using the Aubin-Nitsche duality argument, following the proof of~\cite[Proposition 11]{BBH09}.

\section{Estimate of the condition number}
Let $\{\bfvarphi_i\}_{i=1}^{N_1}$ and $\{\chi_i\}_{i=1}^{N_2}$ be a standard finite element basis in $\mcW_h$ and $\mcV_h$, respectively. Let $\mcA$ be the stiffness matrix associated with the formulation~\eqref{eq:dg}. Matrix $\mcA$ has dimension $(N_1+N_2)\times (N_1+N_2)$. For the Euclidian norm of a vector $X \in \mathbb{R}^N$ we use the notation $|X|^2_N = \sum_{i=1}^N X_i^2$.
We recall that the spectral condition number $\kappa(\mcA)$ is defined by
\begin{equation}
\kappa(\mcA) = | \mcA |_N |\mcA^{-1} |_N.
\end{equation}
Here $N=(N_1+N_2)$ and $|\mcA|_N = \sup_{|X|_N =1} |\mcA X|_N$ for $\mcA \in \mathbb{R}^{N\times N}$. The expansion 
$\bfu_h = \sum_{i=1}^{N_1} U_i \bfvarphi_i$ and $p_h=\sum_{i=1}^{N_2} P_i \chi_i$ define isomporphisms that map $\bfu_h \in \mcW_h$ to $U \in {\mathbb R}^{N_1}$ and 
$p_h \in \mcV_h$ to $P \in {\mathbb R}^{N_2}$, respectively. We have for $V \in \mathbb{R}^{N}$ being the concatenation of $U$ and $P$ the following estimate 
\begin{equation}\label{eq:rneqv}
c_1 h^{-1} \left( \| p_h \|_{0,\omegahO\cup \omegahT} + \| \bfu_h \|_{0,\omegahO\cup \omegahT} \right) \leq | V |_N \leq c_2 h^{-1}  \left( \| p_h \|_{0,\omegahO\cup \omegahT} + \| \bfu_h \|_{0,\omegahO\cup \omegahT} \right). 
\end{equation}

To derive an estimate of the condition number we first prove a Poincare type inequality in
Lemma \ref{lemmaCondB} and an inverse estimate in Lemma \ref{lemmaCondC}. Then the 
condition number estimates follows from these two lemmas and the approach in \cite{EG04}.

\begin{lemma}\label{lemmaCondB} {If the solution to the dual problem
\begin{equation}\label{eq:dualp}
-\nabla \cdot \bfeps(\bfphi) = \bfv_h \text{ in $\Omega$},\quad \nabla \cdot \bfphi=0  \text{ in $\Omega$}, \ \bfphi =0 \text{ on $\partial \Omega$},  
\end{equation}
satisfy the elliptic regularity estimate
 \begin{equation}\label{eq:ellipticreg}
 \|\bfphi \|^2_{2,\Omega}\leq C_\Omega \|\bfv_h\|^2_{0,\Omega}.
 \end{equation} 
Then the following estimate holds
\begin{align}\label{poincare}
 \left( \| q_h \|_{0,\omegahO\cup \omegahT} + \| \bfv_h \|_{0,\omegahO\cup \omegahT} \right) &\leq C\max(\mu_{\max}^{1/2},\mu_{\min}^{-1/2}) \tn (\bfv_h,q_h) \tn_h
\end{align}
for all $(\bfv_h,q_h) \in \mcW_h \times \mcV_h$ where C is a positive constant}.
\end{lemma}
\begin{proof}
We have that $ \| q_h \|_{0,\omegahO\cup \omegahT} \leq \mu_{\max}^{1/2}  \| \mu^{-1/2}q_h \|_{0,\omegahO\cup \omegahT}$, and
we need to show that $\| \bfv_h \|_{0,\omegahO\cup \omegahT} \leq C\mu_{\min}^{-1/2} \tn \bfv_h \tn_h$.
Multiplying the dual problem (\ref{eq:dualp}) with $\bfv_h$, integrating by 
parts, and using the Cauchy-Schwarz inequality we get
\begin{align}
&\| \bfv_h\|^2_{0,\Omega_1\cup \Omega_2} =(\bfeps( \bfv_h),\bfeps( \bfphi))_{0,\Omega_1 \cup \Omega_2} - (\ldb \bfv_h \rdb, \bfeps( \bfphi) \bfn )_{\Gamma}  - (\bfv_h, \bfeps(\bfphi) \bfn )_{\partial \Omega}\leq
\nonumber \\
&C\left(\| \bfeps(\bfv_h) \|^2_{0,\Omega_1\cup \Omega_2} + \|[\bfv_h]\|^2_{1/2,h,\Gamma}  + \|\bfv_h\|^2_{1/2,h,\partial \Omega} \right)^{1/2} \left(\|\bfeps( \bfphi) \|^2_{0,\Omega_1\cup\Omega_2} + \|\bfeps( \bfphi) \bfn \|^2_{-1/2,h,\Gamma} \right)^{1/2}.
\end{align}
Using that $\| \bfeps( \bfphi) \bfn  \|^2_{0,\Gamma}\leq C_\Gamma\|\bfphi \|^2_{2,\Omega}$ and the elliptic regularity estimate~\eqref{eq:ellipticreg} we 
have 
\begin{align}
&\|\bfeps( \bfphi) \|^2_{0,\Omega_1\cup\Omega_2} + \|\bfeps (\bfphi) \bfn  \|^2_{-1/2,h,\Gamma}  \leq 
C\|\bfphi \|^2_{2,\Omega_1\cup \Omega_2}\leq C \|\bfv_h\|^2_{0,\Omega_1\cup \Omega_2}.
\end{align}
Thus, 
\begin{align}
\| \bfv_h\|_{0,\Omega_1\cup\Omega_2} \leq C\left(\| \bfeps(\bfv_h) \|^2_{0,\Omega_1\cup \Omega_2} + \|[\bfv_h]\|^2_{1/2,h,\Gamma}  + \|\bfv_h\|^2_{1/2,h,\partial \Omega}\right)^{1/2}\leq C\mu_{\min}^{-1/2}\tn \bfv_h \tn. 
\end{align}
Following the proof of Lemma~\ref{lem:technical} we can show that
\begin{align}\label{eq:normeqv}
\sum_{i=1}^2  \|\bfv_{h,i} \|_{0,\omegahi} &\leq C\left( \|  \bfv_h \|_{0,\Omega_1\cup\Omega_2} + \left(\sum_{i=1}^2 \sum_{F\in\mcF_{\Gamma,i}} h^s \| \ldb \bfn_F \cdot\nabla \bfv_{h,i} \rdb_F \|_{0,F}^2 \right)^{1/2}\right).
\end{align}
Finally, we have that 
\begin{align}
\sum_{i=1}^2  \| \bfv_{h,i} \|_{0,\omegahi} & \leq  C\left( \mu_{\min}^{-1/2} \tn \bfv_h \tn + \mu_{\min}^{-1/2} (J_{\bfu}(\bfv_h, \bfv_h))^{1/2} \right)
\leq C\mu_{\min}^{-1/2} \tn \bfv_h \tn_h.
\end{align}
Recalling the definition of the norm $\tn (\cdot, \cdot) \tn_h$, the claim follows. 
\end{proof}

\begin{lemma}\label{lemmaCondC} The following estimate holds
\begin{align}\label{inverse}
 \tn (\bfv_h,q_h) \tn_h &\leq C h^{-1}  \max(\mu_{\max}^{1/2},\mu_{\min}^{-1/2}) \left( \| q_h \|_{0,\omegahO\cup \omegahT} + \| \bfv_h \|_{0,\omegahO\cup \omegahT} \right)
\end{align}
for all $(\bfv_h,q_h) \in \mcW_h \times \mcV_h$  {where C is a positive constant}.
\end{lemma}
\begin{proof}
Note that $  \| \mu^{-1/2}q_h \|_{0,\omegahO\cup \omegahT} \leq   \mu_{\min}^{-1/2} \| q_h \|_{0,\omegahO\cup \omegahT}$. Using the Cauchy-Schwarz inequality and the trace inequalities ~\eqref{eq:inversineq}--\eqref{eq:traceineqF} we have
\begin{align}
J_p(q_h,q_h)&\leq \sum_{i=1}^2 \mu_i^{-1}\sum_{F \in \mcF_{\Gamma,i}} C h^3 \| \ldb \nabla q_{h,i} \rdb_F \|_{0,F}^2\leq  \mu_{\min}^{-1}Ch^{2}\sum_{i=1}^2 \sum_{K\in\mcK_i} \| \nabla q_{h,i}\|_{0,K}^2 \nonumber \\
&\quad \leq  \mu_{\min}^{-1}C\sum_{i=1}^2 \sum_{K\in\mcK_i} \| q_{h,i}\|_{0,K}^2 \leq \mu_{\min}^{-1}C\| q_h \|_{0,\omegahO\cup \omegahT}^2. 
\end{align}
In the same way we obtain
\begin{align}
J_{\bfu}(\bfv_h,\bfv_h)&\leq  \mu_{\max}Ch^{s-3}\| \bfv_h \|_{0,\omegahO\cup \omegahT}^2. 
\end{align}
The standard inverse inequality~\eqref{eq:inversineq} yields $\| \mu^{1/2}\bfeps(\bfv_h)\|_{0,\Omega_1\cup\Omega_2} \leq Ch^{-1}\mu_{\max}^{1/2}\|\bfv_h\|_{0,\Omega_1\cup\Omega_2}$.
The Lemma follows using the trace inequalities~\eqref{eq:traceineqF} and~\eqref{eq:traceineq} on each of the interface and boundary contributions to $\tn \bfv_h \tn$.
\end{proof}

\begin{theorem} The following estimate of the spectral condition number of the stiffness matrix holds
\begin{equation}
\kappa( \mcA )\leq C \max(\mu_{\max}^{2},\mu_{\min}^{-2})h^{-2},
\end{equation}
 {where C is a positive constant}.
\end{theorem}
\begin{proof} We need to estimate $| \mcA |_N$ and $|\mcA^{-1}|_N$. Let $V \in \mathbb{R}^{N}$ and $W \in \mathbb{R}^{N}$ be the vectors containing the coefficients corresponding to $(v_h,q_h) \in \mcW_h\times \mcV_h$ and $(w_h,r_h) \in \mcW_h\times \mcV_h$, respectively.
Starting with $|\mcA |_N$ we have
\begin{align}
|\mcA V |_N &= \sup_{ W  \in {\bf R}^N } \frac{(\mcA  V, W)_N}{| W  |_N}
\nonumber \\
&= \sup_{(\bfw_h,r_h) \in \mcW_h \times \mcV_h }  \frac{A_h(\bfv_h,q_h;\bfw_h,r_h)+\varepsilon_{\bfu}J_{\bfu}(\bfv_h,\bfw_h)+\varepsilon_{p}J_p(q_h,r_h)}{| W |_N}.
\end{align}
We now use the continuity of $A_h(\cdot,\cdot;\cdot,\cdot)$ established in Lemma \ref{lem:contb} together with that $\tn (\bfv_h,q_h) \tn \leq C_* \tn(\bfv_h,q_h)\tn_h$ and the Cauchy-Schwarz inequality to obtain
\begin{align}
A_h(\bfv_h,q_h;\bfw_h,r_h)&+\varepsilon_{\bfu}J_{\bfu}(\bfv_h,\bfw_h)+\varepsilon_{p}J_p(q_h,r_h)\leq \nonumber \\
&C\max{(C_*C_{A},\varepsilon_{\bfu},\varepsilon_{p})} \tn (\bfv_h,q_h) \tn_h \tn (\bfw_h,r_h) \tn_h.
\end{align}
Lemma~\ref{lemmaCondC} and equation~\eqref{eq:rneqv} then yield
\begin{align}
&A_h(\bfv_h,q_h;\bfw_h,r_h)+\varepsilon_{\bfu}J_{\bfu}(\bfv_h,\bfw_h)+\varepsilon_{p}J_p(q_h,r_h)\leq \nonumber \\
&C\max(\mu_{\max}^{},\mu_{\min}^{-1}) h^{-2}\left( \| q_h \|_{0,\omegahO\cup \omegahT} + \| \bfv_h \|_{0,\omegahO\cup \omegahT} \right)\left( \| r_h \|_{0,\omegahO\cup \omegahT} + \| \bfw_h \|_{0,\omegahO\cup \omegahT} \right)
 \nonumber \\
& \leq C\max(\mu_{\max}^{},\mu_{\min}^{-1}) | V |_N| W |_N.
\end{align}
Thus, we have the estimate
\begin{equation}\label{Aest}
|\mcA |_N=\sup_{V\in \mathbb{R}^{N}}\frac{|\mcA V |_N}{| V |_N} \leq C\max(\mu_{\max}^{},\mu_{\min}^{-1}).
\end{equation}
Next we turn to the estimate of $|\mcA^{-1}|_N$. Using equation (\ref{eq:rneqv}), Lemma~\ref{lemmaCondB}, and the inf-sup stability (Theorem~\ref{thm:infsup}) we get
\begin{align}
|V|_N &\leq C h^{-1} \left( \| q_h \|_{0,\omegahO\cup \omegahT} + \| \bfv_h \|_{0,\omegahO\cup \omegahT} \right)
\nonumber \\
&\leq C\max(\mu_{\max}^{1/2},\mu_{\min}^{-1/2}) h^{-1} \tn (\bfv_h,q_h) \tn_h
\nonumber \\
&\leq C\max(\mu_{\max}^{1/2},\mu_{\min}^{-1/2}) h^{-1}  \sup_{ (\bfw_h,r_h) \in \mcW_h \times \mcV_h }  \frac{A_h(\bfv_h,q_h;\bfw_h,r_h)+\varepsilon_{\bfu}J_{\bfu}(\bfv_h,\bfw_h)+\varepsilon_{p}J_p(q_h,r_h)}{\tn (\bfw_h,r_h) \tn_h}  
\nonumber \\
&\leq C\max(\mu_{\max}^{1/2},\mu_{\min}^{-1/2}) h^{-1}  \sup_{W \in {\bf R}^N} \frac{(A V,W)_N}{|W|_N} \frac{|W|_N}{\tn (\bfw_h,r_h)\tn_h}\nonumber \\
&\leq  C\max(\mu_{\max}^{1/2},\mu_{\min}^{-1/2}) h^{-2} |AV|_N \frac{\left( \| r_h \|_{0,\omegahO\cup \omegahT} + \| \bfw_h \|_{0,\omegahO\cup \omegahT} \right)}{\tn (\bfw_h,r_h)\tn_h}.
\end{align}
We conclude that $|V|_N \leq C\max(\mu_{\max},\mu_{\min}^{-1}) h^{-2}|AV|_N$. Setting $V = A^{-1} W$ we obtain
\begin{equation}\label{Ainvest}
|A^{-1}|_N \leq C \max(\mu_{\max}^{1},\mu_{\min}^{-1})h^{-2}.
\end{equation}
Combining estimates (\ref{Aest}) and (\ref{Ainvest}) of $|A|_N$ and $|A^{-1}|_N$ the theorem follows.
\end{proof}

\section{Numerical examples}\label{sec:NumEx}
We have shown that the proposed finite element method is of optimal
convergence order and results in a well-conditioned equation system.
In this section we present results for numerical experiments in two space
dimensions using the proposed method (see Section~\ref{sec:FEM}). We study the convergence rate of the numerical solution and the condition number of the system matrix for three examples.   
A direct solver is used to solve the linear systems. 

{The interface is in general not available exactly, instead we have to use some kind of discrete 
representation $\Gamma_h$ of $\Gamma$. Our method is independent of the particular type of representation 
of the interface but here we use the standard level set method. We define a piecewise linear approximation to the distance function on the velocity mesh and the interface is  approximated as the zero level set of this approximate distance function. The interface is thus represented by linear segments on $\mcK_{h/2}$ which results in an $\mathcal{O}(h^2)$ approximation of the interface $\Gamma$.  The errors we report in the numerical examples below are all computed on the domains $\Omega_1$ and $\Omega_2$ that are separated by the discrete interface $\Gamma_h$.}

Unless stated otherwise, we report the size of the velocity mesh
$h_x$. The pressure mesh is twice as coarse. The parameters $\kappa_1$
and $\kappa_2$ are chosen according to expression~\eqref{eq:kappa} and
the penalty parameter $\lambda_\Gamma$ is chosen locally according to
expression~\eqref{eq:lambda}. Dirichlet conditions for the velocity
are imposed weakly and the penalty parameter $\lambda_{\partial
  \Omega}$ enforcing the boundary conditions is chosen according to
equation~\eqref{eq:lambdaB}. {The condition $(\mu^{-1}p_h,1)_{\Omega_1\cup\Omega_2}=0$ is imposed using a Lagrange multiplier.}

Both of the stabilization terms $J_p(p_h,q_h)$ and $J_{\bfu}(\bfu_h,\bfv_h)$ are needed in order to have control of the condition number.  In all the examples, the stabilization parameter $\varepsilon_p=1$ and $\varepsilon_{\bfu}=10^{-3}$. The errors are not sensitive to these parameters. Also, recall that we have defined $\mu=2\mu_i$ in $\Omega_i$, $i=1,2$.

\subsection{Example 1: A continuous problem}
We consider a continuous problem presented in~\cite{BBH09}. The
computational domain is $[0, 1] \times [0, 1]$, the interface is a
circle centered in $(0.5,0.5)$ with radius $0.3$ and $\mu=2$.  The
Dirichlet boundary conditions on $\partial \Omega$ are chosen such
that the exact solution is given by $\bfu=(20xy^3, 5x^4-5y^4)$ and
$p=60x^2y-20y^3-5$.

In this example we use a regular mesh. We choose
$\lambda_{\partial \Omega}$ according to equation~\eqref{eq:lambdaB} with $G$, and $H$ such that $\lambda_{\partial \Omega}=\frac{15}{h_x}$. Furthermore, we take $C=2$ and $D=0.05$ in the expression for the penalty parameter $\lambda_\Gamma$ (equation~\eqref{eq:lambda}). The condition number of the system matrix and the error depends on these constants. However, we have not optimized these constants. 

\begin{figure}
 \begin{center}
    \includegraphics[width=0.49\textwidth]{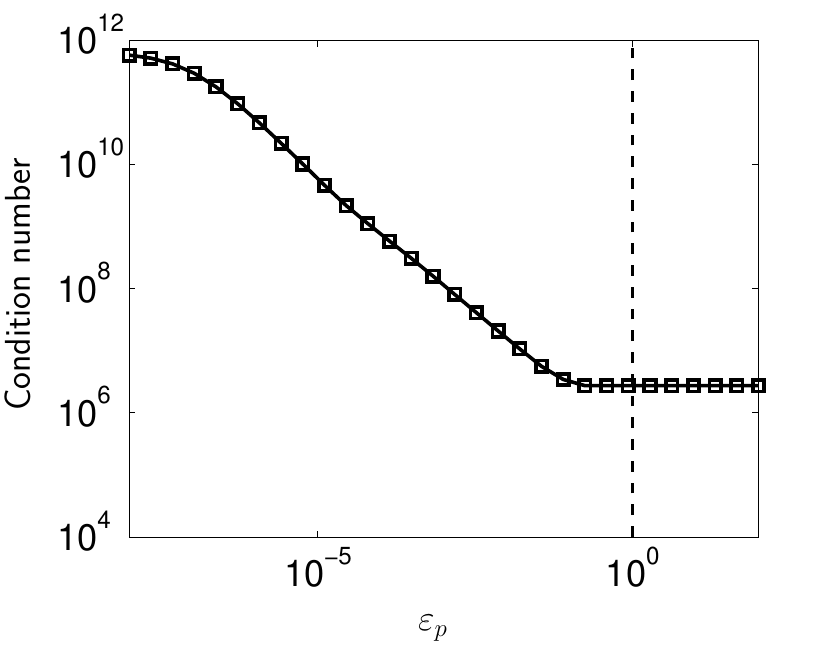} 
    \includegraphics[width=0.49\textwidth]{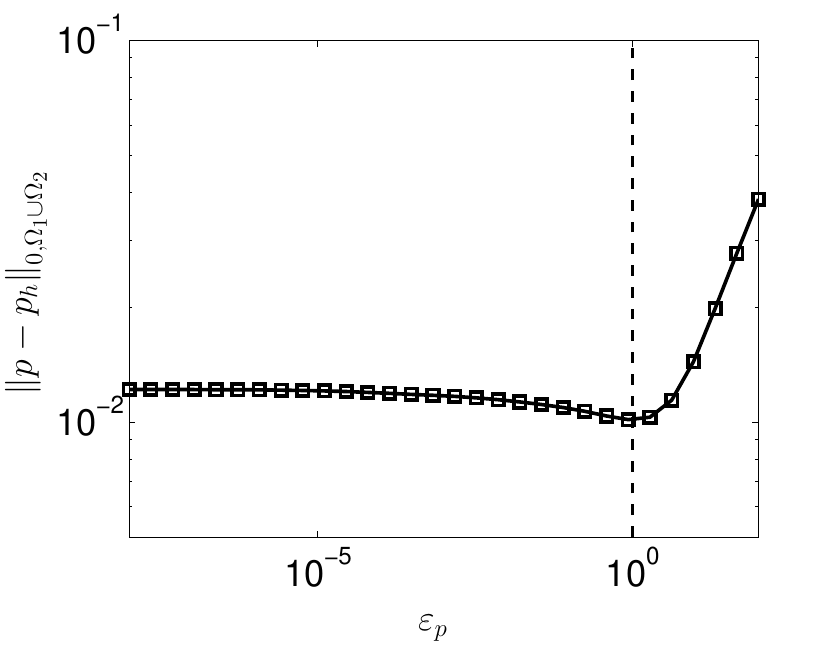} 
  \end{center}
\caption{Spectral condition number and the error in the pressure versus the
  stabilization constant $\varepsilon_p$ for the continuous
  problem when $h_x=0.0125$.  Left panel: The estimated
  spectral condition number versus $\varepsilon_p$.  Right panel: The
  error in the pressure measured in the $L^2$ norm versus
  $\varepsilon_p$. The dashed line indicates the value of $\varepsilon_p$ that we have used in the computations.  \label{fig:errorvsstabp}}
\end{figure}
In Fig.~\ref{fig:errorvsstabp} we show the spectral
condition number and the error in the pressure as a function of the
stabilization parameter $\varepsilon_p$ for $h_x=0.0125$. As seen in
the figure the condition number of the system increases as
$\varepsilon_p$ decreases. Also, a too small $\varepsilon_p$ results
in a condition number that increases rapidly as the mesh size is
reduced. However, for $\varepsilon_p\leq 1$, the error is not sensitive to the stabilization. Therefore, we have chosen $\varepsilon_p=1$ in our computations. In this example the results for
$\varepsilon_{\bfu}=10^{-3}$ and $\varepsilon_{\bfu}=0$
coincide. Since the interface is not very close to any meshlines we have
control of the condition number even when $\varepsilon_{\bfu}=0$. This
is not the case in the last example in this section.

The convergence for the velocity and the pressure in the $L^2$ norm is
shown in Fig.~\ref{fig:Excont}.  Since in this example neither the
pressure nor the velocity have discontinuities we compare our method
with the standard continuous finite element method. Compared to using
standard continuous finite element methods we obtain just slightly
larger errors for the pressure. However compared to the method
in~\cite{BBH09} (see Fig.~3 in~\cite{BBH09} ) we obtain much smaller
errors for the pressure. In Fig.~\ref{fig:Excont} we see the optimal
second order convergence for the velocity in the $L^2$ norm but for
the pressure we observe better convergence than the
expected first order measured in the $L^2$ norm.
\begin{figure}
 \begin{center} 
    \includegraphics[width=0.49\textwidth]{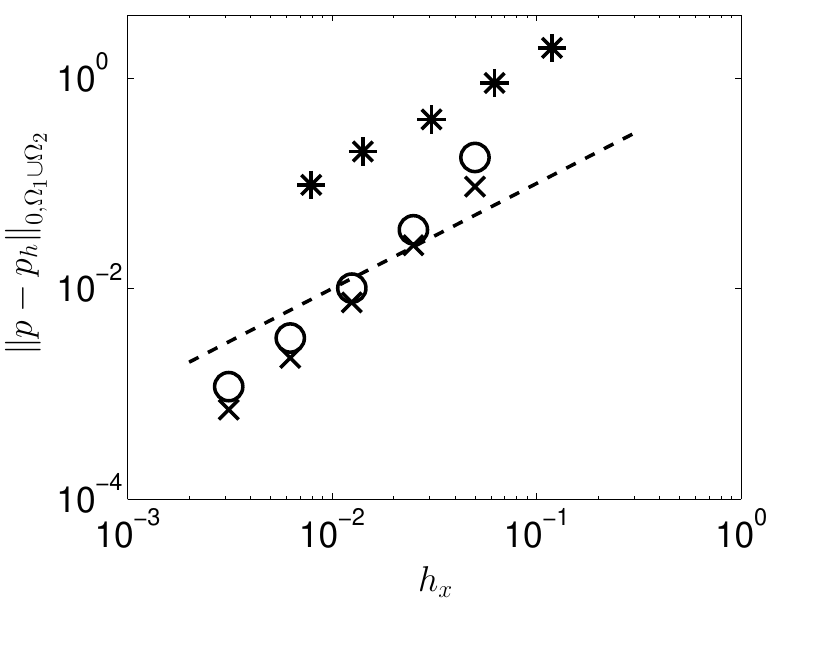} \hfill
    \includegraphics[width=0.49\textwidth]{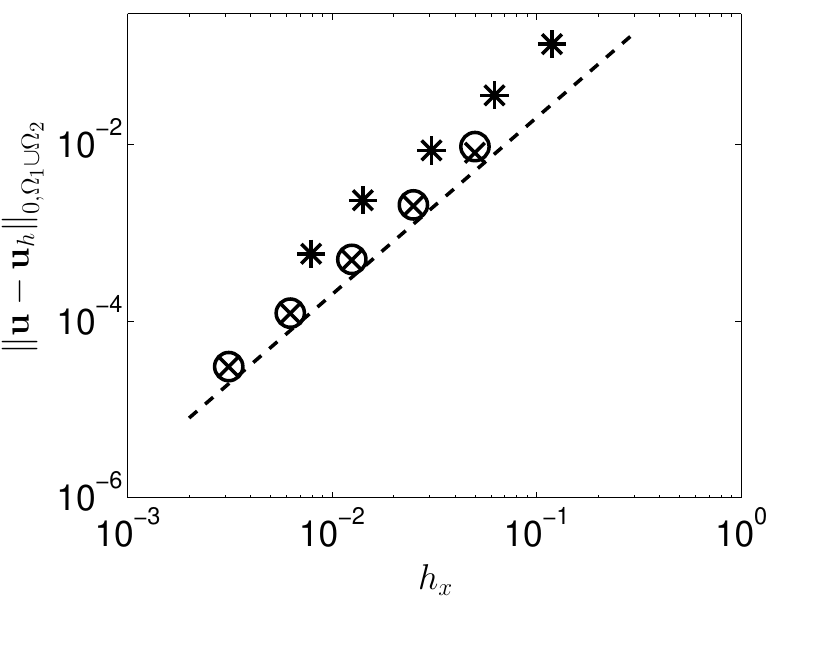}
  \end{center}
\caption{Convergence rate of the error in pressure and velocity for
  the continuous problem. Circles ($\circ $),
  Crosses ($\times$), and stars ($*$) represent the new method, the
  standard continuous FEM, and the method in~\cite{BBH09},
  respectively. Left panel: The error in the pressure measured in the
  $L^2$ norm versus the mesh size $h_x$. The dashed line $y=h_x$, shows
  the expected convergence order.  Right panel: The error in the
  velocity measured in the $L^2$ norm versus the mesh size $h_x$. The
  dashed line $y=2h_x^{2}$, indicates the optimal convergence
  order. \label{fig:Excont}}
\end{figure} 
\begin{figure}
 \begin{center}
   \includegraphics[width=0.49\textwidth]{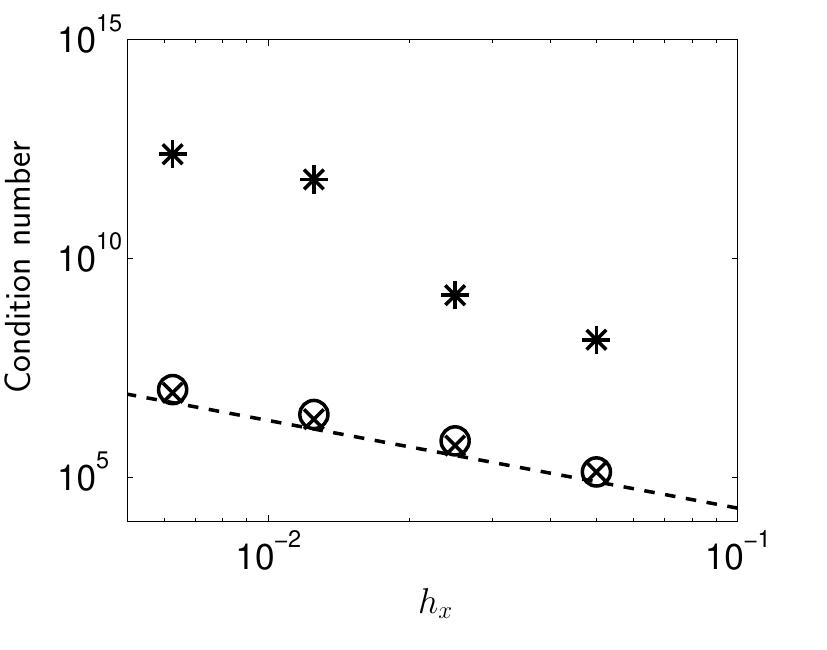} \hfill
 \includegraphics[width=0.49\textwidth]{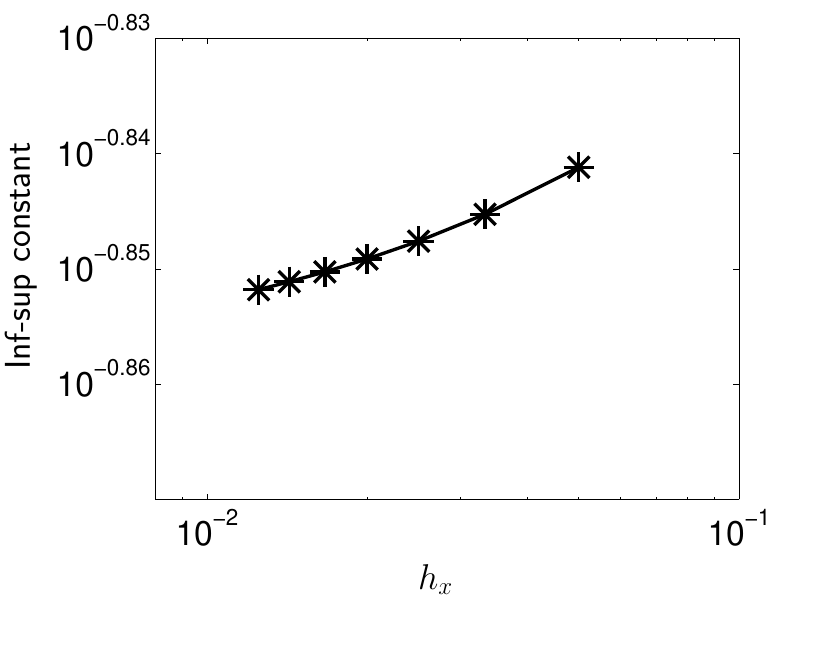}
  \end{center}
\caption{Spectral condition number and the inf-sup constant.  Circles
  ($\circ $), Crosses ($\times$), and stars ($*$) represent the
  presented method with $\varepsilon_p=10^{-2}$, the standard
  continuous FEM, and the unstabilized method
  (i.e. $\varepsilon_p=0$), respectively. Left panel: Condition number
  versus mesh size; the dashed line shows the slope of the condition number $O(h_x^{-2})$ of standard FEM. 
  Right Panel: The inf-sup constant versus mesh
  size. \label{fig:condinfsupEx1}}
\end{figure}
In Fig.~\ref{fig:condinfsupEx1} we show the spectral condition
number. The condition number using the proposed stabilized method
grows as $\mathcal{O}(h^{-2})$ just as it does for the standard finite
element method, which is optimal. For a fixed mesh size the condition
number of the system matrix using the proposed method is very close to
the condition number of the system matrix using the standard
continuous finite element method. We see that the condition number
grows erratically with the mesh size when there is no stabilization
for the pressure, i.e., $\varepsilon_p=0$. However, we also see in the
figure that the numerically estimated inf-sup constant in case
$\varepsilon_p=0$ is essentially independent of the mesh size. Thus, our numerical
results suggest that the inf-sup condition is satisfied in this case even when
there is no stabilization.

\subsection{Example 2: Static drop}
Consider a circular interface $\Gamma$ of radius R in equilibrium in the
interior of a domain in two dimensions with $\mu=2$, $\sigma=1$
and vanishing $u_b$ on $\partial\Omega$. The exact solution is
$u\equiv0,\, p_1=0,\,p_2=\sigma/R$. This corresponds to a circular
fluid drop in equilibrium with the surrounding fluid.  

{In this example our method is compared to standard continuous finite elements with two common representations of the surface tension force.  A common strategy in fixed-grid methods is to include the jump conditions in the model by adding a singular source term to the equations of motion expressed in terms of a Dirac delta function with
support on the interface. Numerically, delta functions can be approximated by regularized discrete operators that distribute the
force over a band near the interface. We refer to this approach as a regularized surface tension representation.  Instead of regularizing the delta function an alternative in the finite element framework is  to evaluate a line integral.  We refer to this approach as a sharp surface tension representation. Imbalances between the discrete representation of the surface tension force and the pressure jump leads to a nonzero velocity field. We will refer to these unphysical flows as spurious currents.}

In Figure~\ref{fig:pressure} we compare the pressure approximation using
the new method with the results obtained in~\cite{ZKK11}. In this case
$R=0.5$ and we prescribe the exact curvature $\kappa=2$.

We use a regular mesh with $h_x=0.025$ in the velocity mesh. We choose
$\lambda_{\partial \Omega}$ and $\lambda_\Gamma$ as in the previous
example. From Table~\ref{tab:statD} we see that for the new method the
magnitude of spurious currents and the error in the pressure are of
the order of machine epsilon.  However, using a sharp surface tension
representation and standard continuous finite element methods the
magnitude of spurious currents are large and may lead to unphysical
movements of the interface.  With standard globally continuous finite
element methods the pressure either oscillates or is smeared out
depending on if a sharp or regularized surface tension representation
is used, see the two leftmost panels of Fig.~\ref{fig:pressure}.
With the new method the discontinuous pressure is accurately
represented even on coarse meshes.
\begin{rem}\label{rem:staticD}
 We would like to emphasize that in order to get the magnitude of spurious currents and the error in the pressure of the order of machine epsilon even on coarse meshes it is important to use the bilinear form $b_h^2$ (equation~\eqref{eq:bgradp}). 
{This can be understood by inserting the exact solution $u=0$ and $f=0$ into the variational form~\eqref{eq:dg}, which yields  
\begin{equation}
-b_h(\bfv_h,p_h)+ \varepsilon_p J_p(p_h,q_h) = (\sigma \kappa , \langle \bfv_h \cdot \bfn \rangle)_{\Gamma}  \quad \forall (\bfv_h, q_h) \in \mcW_h \times \mcV_h. 
\end{equation}
The two forms $b_h^1$ and $b_h^2$ in equation~\eqref{eq:bdivu} and~\eqref{eq:bgradp}, respectively are mathematically equivalent, however $b_h^2$ contains a term $(\ldb  p_h \rdb  ,  \langle \bfv_h \cdot \bfn \rangle)_\Gamma$ which is in balance with the term $(\sigma \kappa , \langle \bfv_h \cdot \bfn \rangle)_{\Gamma}$  on the right hand side. Thus, we obtain a perfect balance between the terms on the left and the right hand sides of the variational form when $b_h^2$ is used. 
Using $b_h^1$ in equation~\eqref{eq:bdivu} results in spurious currents and errors in the pressure but the errors decrease with mesh refinement. 
}
\end{rem}
\begin{table} 
\begin{center} 
\begin{tabular}{|c|c|c|r|}
 \hline
& & & \\ [-1.5ex] 
  
   & $\|  \bfu_h \|_{\infty}$ & $\| p-p_h \|_{\infty}$ & Condition number \\ [0.15ex]  \hline & & & \\ [-1.8ex] 
   Regularized force &  $\mathcal{O}(10^{-16})$ &  $1.0129$  & $5.36 \cdot 10^{5}$\\[0.15ex]  \hline & & & \\ [-1.8ex] 
   Sharp force & $0.0126$ & $1.0164$ & $5.36 \cdot 10^{5}$ \\[0.15ex]  \hline & & & \\ [-1.8ex] 
   New method &  $\mathcal{O}(10^{-16})$ &  $\mathcal{O}(10^{-16})$  & $6.74 \cdot 10^{5}$ \\  \hline
\end{tabular}
\caption{Spurious velocities, error in the pressure approximation, and
  the spectral condition number for the static drop. Standard
  continuous finite elements with a regularized and a sharp
  approximation of the surface tension force are compared with the new
  method with $\varepsilon_p=10^{-1}$. \label{tab:statD}}
\end{center}
\end{table}
\begin{figure}
\begin{center}
\includegraphics[width=0.3\textwidth]{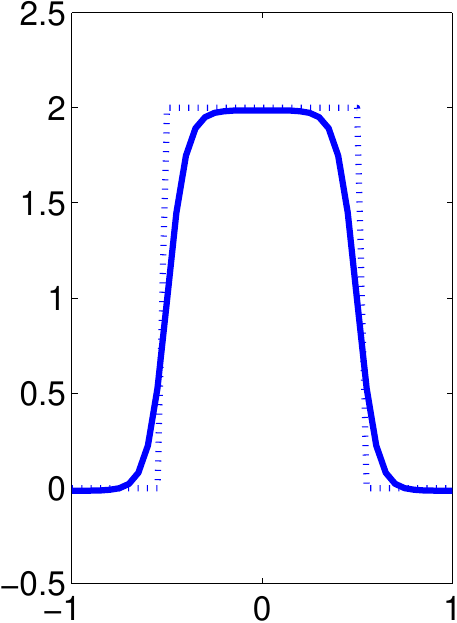} \hfill
\includegraphics[width=0.3\textwidth]{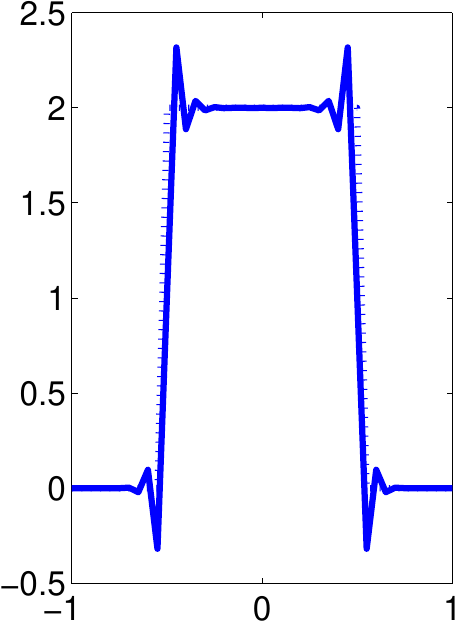} \hfill
\includegraphics[width=0.3\textwidth]{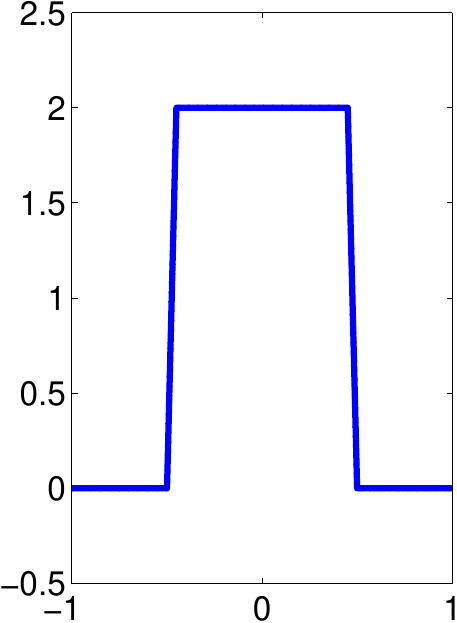}
\caption{Cross section of the pressure approximation for the static
  drop.  The exact curvature $\kappa=2$ is prescribed. The dotted
  lines in all figures represent the exact pressure.  Left panel: A
  Standard continuous finite element method and a regularized surface
  tension force is used.  Middle panel: A Standard continuous finite
  element method is used with sharp representation of the surface
  tension force. Right panel: The pressure is approximated using the
  new finite element method. \label{fig:pressure}}
\end{center}
\end{figure}

{\subsection{Example 3: A discontinuous problem}
We now consider a problem where the pressure is discontinuous and the velocity field has a kink at the interface due to different fluid viscosities. The interface is the straight line $y=0$ and the jump condition $\ldb \mu \mathbf{D}(\mathbf{u})\cdot \mathbf{n}-p \mathbf{n} \rdb \cdot \mathbf{n}= 10$ is imposed at the interface. The
viscosity 
\begin{equation}
\mu=\left\{ \begin{array}{ll}
2 & y>0, \\
200 & y<0
\end{array} \right.
\end{equation}
and $\textbf{f}=(2x,4x)$. The computational domain is $[0, 4] \times [-0.4, 0.6]$ and the Dirichlet boundary conditions for the velocity are chosen such
that the exact solution is given by
\begin{align}
\bfu(x,y)&=(u^x(x,y), u^y(x,y))=\lp \frac{x^2y}{\mu},\frac{-xy^2}{\mu} \rp, \nonumber \\
p(x,y) &= 2xy+x^2 +10\mathcal{X}(y),
\end{align}
where $\mathcal{X}(y)=1$, if $y<0$ and zero otherwise. Note that the pressure and the velocity field are not in our cutFEM space.  
The interface intersects the domain boundary. The penalty parameter $\lambda_{\partial \Omega}$ is chosen according to equation~\eqref{eq:lambdaB} with $G=0.005$ and $H=8.04$ at elements that are also cut by the interface and otherwise $H=4.02$. The constants in $\lambda_{\Gamma}$ are chosen as $C=3.5$ and $D=0.05$. The condition number depends on these constants but we have not optimized these numbers.} 
\begin{figure}
\begin{center} 
\includegraphics[width=0.85\textwidth]{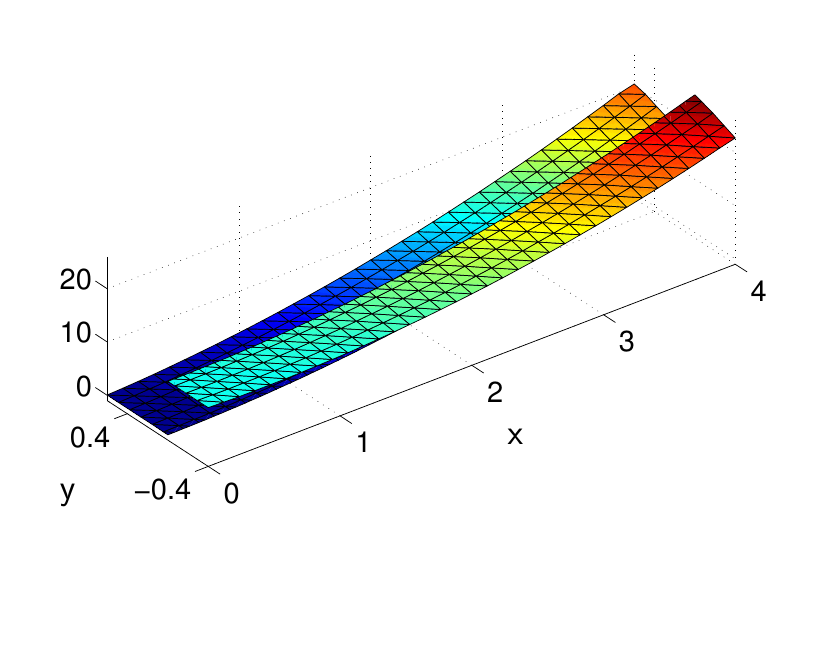}
\includegraphics[width=0.85\textwidth]{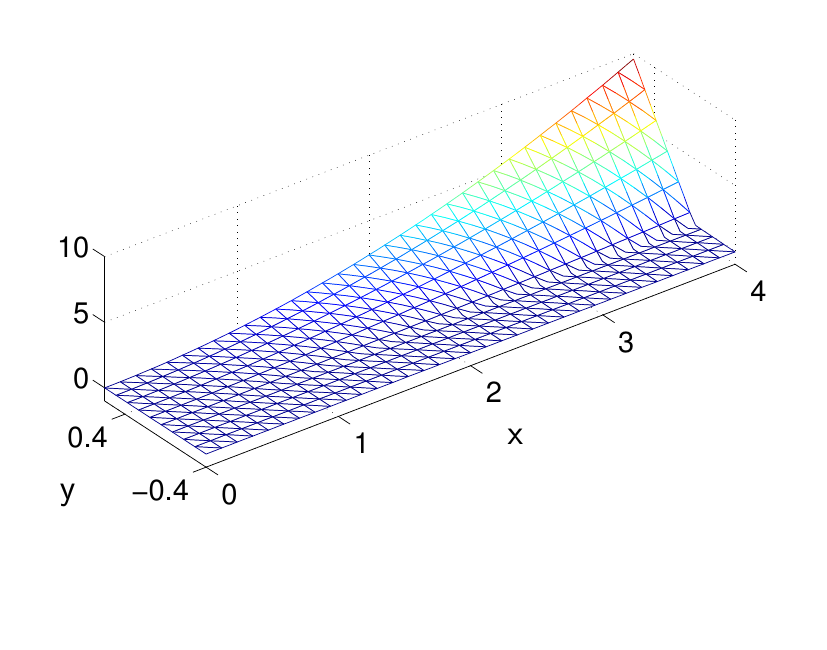}
\caption{Upper Panel: Approximation of the discontinuous pressure in
  Example 3.  Lower Panel: Approximation of the weakly discontinuous velocity component $u^x$ in
  Example 3.  The mesh does not coincide with the interface. There
  are 35 grid points along the x-axis in the pressure mesh.\label{fig:soldiffvisc}}
\end{center}
\end{figure}
\begin{figure}
\begin{center}
\includegraphics[width=0.5\textwidth]{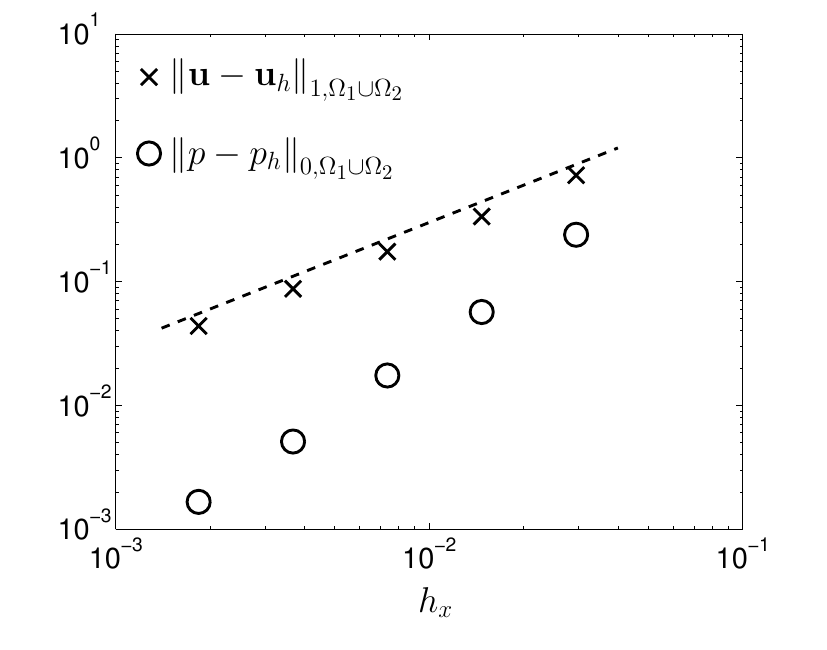}
\caption{ The error in the pressure measured in the $L^2$ norm and the
  error in velocity measured in the $H^1$ norm versus the mesh size
  $h_x$. The dashed line represents
  $y=30h_x$.\label{fig:errordiffvisc}}
\end{center}
\end{figure}
\begin{figure}
\begin{center} 
\includegraphics[width=0.49\textwidth]{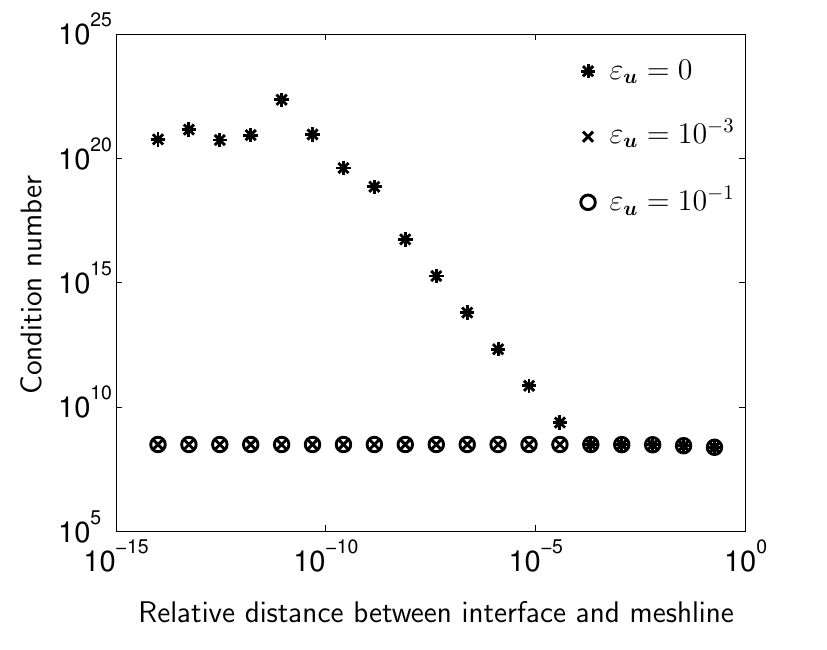}
\includegraphics[width=0.49\textwidth]{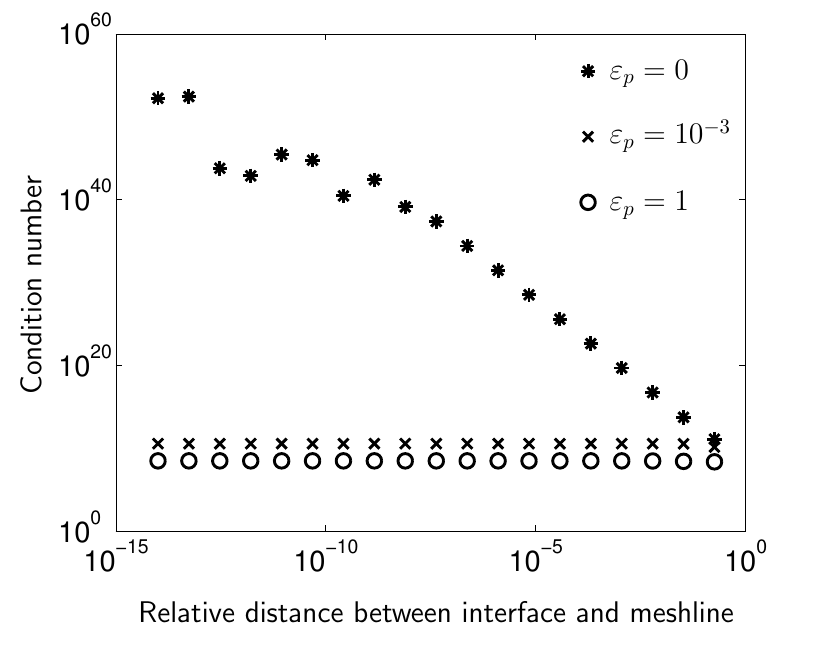}
\caption{The spectral condition number as a function of the relative distance between the
  interface and the mesh line for different values of $\varepsilon_{p}$ and $\varepsilon_{\bfu}$. Left Panel:  $\varepsilon_p=1$. Right Panel: $\varepsilon_{\bfu}=10^{-3}$. The mesh size is fixed, there are 18 grid points along the x-axis in the pressure mesh.
\label{fig:condvsdist}}
\end{center}
\end{figure}

{In Fig.~\ref{fig:soldiffvisc} we show the approximation of the discontinuous pressure and the weakly discontinuous velocity component $u^x$ using the proposed finite element method. The error in the velocity measured in the $H^1$ norm and the error in the pressure measured in the $L^2$ norm are shown for different mesh sizes in Fig.~\ref{fig:errordiffvisc}.  We have as expected first order convergence for the velocity in the $H^1(\Omega_1\cup \Omega_2)$ norm. For the pressure we observe better than first order convergence in the $L^2(\Omega_1\cup \Omega_2)$ norm.}

{In Fig.~\ref{fig:condvsdist} we show the condition number as a function of the relative distance between the interface and the mesh line for different values of $\varepsilon_{\bfu}$ and $\varepsilon_{p}$. The mesh size is kept fixed with 18 grid points along the x-axis in the pressure mesh. We see that both of the stabilization terms $J_p(p_h,q_h)$ and $J_{\bfu}(\bfu_h,\bfv_h)$ are needed in order to obtain a well conditioned system matrix independently of the location of the interface.  
}

\section{Conclusions}\label{sec:conc}
We have proposed a finite element method which offers a way to accurately solve the Stokes equations involving two immiscible fluids with different viscosities and surface tension. The interface that separates the two fluids can be represented either explicitly, for example as in the immersed boundary method, or implicitly as in the level set method.  Our method allows for discontinuities across the interface which can be located arbitrarily with respect to a fixed background mesh. 

We have used the inf--sup stable P1--iso--P2 element and proven that our method is of optimal-order accuracy, and  that the stabilization terms $J_{\bfu}(\bfu_h,\bfv_h)$ and $J_p(p_h,q_h)$ guarantee that the condition number of the system matrix is $\mathcal{O}(h^{-2})$ independent of the interface location. We expect the method to be applicable also in three space dimensions. For higher-order elements, the stabilization terms $J_{\bfu}(\bfu_h,\bfv_h)$ and $J_p(p_h,q_h)$ will include jumps of derivatives of higher orders, see~\cite{BH11}. One can also include projection operators from~\cite{WZKB} into the stabilization to reduce the amount of stabilization and hence the constant in the error. The method we have presented is simple to {implement} and robust and has properties that are very desirable, in particular for problems with moving interfaces.

\section*{Acknowledgment}
Sara Zahedi is partially supported by the Swedish national strategic e-science research program (eSSENCE).


\begin{thebibliography}{10}
\expandafter\ifx\csname url\endcsname\relax
  \def\url#1{\texttt{#1}}\fi
\expandafter\ifx\csname urlprefix\endcsname\relax\def\urlprefix{URL }\fi
\expandafter\ifx\csname href\endcsname\relax
  \def\href#1#2{#2} \def\path#1{#1}\fi
\bibitem{ASB10}
R.~F. Ausas, F.~S. Sousa, G.~C. Buscaglia, An improved finite element space for
  discontinuous pressures, Comput. Methods Appl. Mech. Engrg. 199 (2010)
  1019--1031.

\bibitem{BBH09}
R.~Becker, E.~Burman, P.~Hansbo, A {Nitsche} extended finite element method for
  incompressible elasticity with discontinuous modulus of elasticity, Comput.
  Methods Appl. Mech. Engrg. 198 (2009) 3352--3360.

\bibitem{BKZ92}
J.~U. Brackbill, D.~Kothe, C.~Zemach, A continuum method for modeling surface
  tension, J.~Comput.~Phys. 100 (1992) 335--353.

\bibitem{B03}
S.~C. Brenner, {P}oincare-{F}riedrichs inequalities for piecewise {$H^1$}
  functions, SIAM J.~Numer.~Anal. 41 (2003) 306--324.

\bibitem{BrSc}
S.~C. Brenner, L.~R. Scott, {The Mathematical Theory of Finite Element
  Methods}, Springer-Verlag, 2008.

\bibitem{BrFo91}
F.~Brezzi, M.~Fortin, Mixed and hybrid finite element methods, Vol.~15 of
  Springer Series in Computational Mathematics, Springer-Verlag, New York,
  1991.


\bibitem{Bu10}
E.~Burman, Ghost penalty, C. R. Math. Acad. Sci. Paris 348~(21-22) (2010)
  1217--1220.

\bibitem{BH12}
E.~Burman, P.~Hansbo, Fictitious domain finite element methods using cut
  elements: {II. A} stabilized {Nitsche} method, Applied Numerical Mathematics
  62 (2012) 328--341.

\bibitem{BH11}
E.~Burman, P.~Hansbo, Fictitious domain methods using cut elements: {III. A}
  stabilized {Nitsche} method for {Stokes'} problem, ESAIM: Math. Model. Numer. Anal., in press,
  DOI:10.1051/m2an/2013123

\bibitem{CB03}
J.~Chessa, T.~Belytschko, An extended finite element method for two-phase
  fluids, J. Appl. Mech. 70 (2003) 10--17.

\bibitem{Dautray}
R.~Dautray, J.-L. Lions, Mathematical analysis and numerical methods for
  science and technology. {V}ol. 2, Springer-Verlag, Berlin, 1988.

\bibitem{EG04}
A.~Ern, J.-L. Guermond, Theory and Practice of Finite Elements, Vol. 159,
  Applied Mathematical Sciences, Springer-Verlag, 2004.

\bibitem{FB10}
T.-P. Fries, T.~Belytschko, The extended/generalized finite element method: An
  overview of the method and its applications, Internat. J. Numer. Methods Engrg. 84
  (2010) 253--304.

\bibitem{GrRe07}
S.~Gross, A.~Reusken, An extended pressure finite element space for two-phase
  incompressible flows with surface tension, J. Comput. Phys. 224 (2007)
  40--58.

\bibitem{HaHa02}
A.~Hansbo, P.~Hansbo, An unfitted finite element method, based on {Nitsche's}
  method, for elliptic interface problems, Comput. Methods Appl. Mech. Engrg.
  191 (2002) 5537--5552.

\bibitem{HGamm}
P.~Hansbo, Nitsche's method for interface problems in computational mechanics,
  GAMM-Mitt. 28~(2) (2005) 183--206.

\bibitem{JoLa13}
A.~Johansson, M.~G. Larson, A high order discontinuous {G}alerkin {N}itsche
  method for elliptic problems with fictitious boundary, Numer. Math. 123~(4)
  (2013) 607--628.

\bibitem{LI06}
Z.~Li, K.~Ito, The Immersed Interface Method: Numerical Solutions of {PDEs}
  Involving Interfaces and Irregular Domains, {SIAM Frontiers in Applied
  Mathematics}, 2006.

\bibitem{MaLa13}
A.~Massing, M.~G. Larson, A.~Logg, Efficient implementation of finite element
  methods on nonmatching and overlapping meshes in three dimensions, SIAM J.
  Sci. Comput. 35~(1) (2013) C23--C47.

\bibitem{Nit}
J.~Nitsche, \"{U}ber ein variationsprinzip zur l\"{o}sung von
  {D}irichlet-problemen bei verwendung von teilr\"{a}umen, die keinen
  randbedingungen unterworfen sind., Abh. Math. Univ. Hamburg 36 (1971) 9--15.

\bibitem{OR06}
M.~A. Olshanskii, A.~Reusken, Analysis of a {S}tokes interface problem, Numer.
  Math. 103 (2006) 129--149.

\bibitem{R08}
A.~Reusken, {Analysis of extended pressure finite element space for two-phase
  incompressible flows}, Comp. Visual. Sci. 11 (2008) 293--305.

\bibitem{WZKB}
E.~Wadbro, S.~Zahedi, G.~Kreiss, M.~Berggren, A uniformly well-conditioned,
  unfitted {Nitsche} method for interface problems, BIT 53 (2013) 791--820.

\bibitem{ZT11}
S.~Zahedi, Numerical Methods for Fluid Interface Problems, Doctoral Thesis in
  Applied and Computational Mathematics, TRITA-CSC-A 2011:07.

\bibitem{ZKK11}
S.~Zahedi, M.~Kronbichler, G.~Kreiss, Spurious currents in finite element based
  level set methods for two-phase flow, Internat. J. Numer. Methods Fluids 69~(9)
  (2012) 1433--1456.

\bibitem{ZuCaCo11}
P.~Zunino, L.~Cattaneo, C. M. Colciago, An unfitted interface penalty method for the
  numerical approximation of contrast problems, Appl. Numer. Math. 61 (2011)
  1059--1076.



\end{thebibliography}

\end{document}